\documentclass[a4paper,10pt]{article}

\usepackage{geometry}
\usepackage{epsfig}
\usepackage{amsmath}
\usepackage{amssymb}
\usepackage{amsthm}
\usepackage{mathrsfs}
\usepackage{color}
\usepackage{amscd}

\usepackage{mathtools} 					
\usepackage{hyperref}						
 \usepackage{booktabs}					

\usepackage{amsthm}
\newtheorem{theorem}{Theorem}[section]

\newtheorem{lemma}[theorem]{Lemma}
\newtheorem{proposition}[theorem]{Proposition}
\newtheorem{corollary}[theorem]{Corollary}
\newtheorem{remark}[theorem]{Remark}
\newtheorem{example}[theorem]{Example}
\newtheorem{examples}[theorem]{Examples}

\theoremstyle{remark}{
}
\theoremstyle{definition}{
\newtheorem{definition}[theorem]{Definition}}

\usepackage{amsfonts}

\newcommand{\hh}{{\mathbb{H}}}

\newcommand{\cc}{{\mathbb{C}}}
\newcommand{\rr}{{\mathbb{R}}}

\newcommand{\D}{\mathbb{D}}

\newcommand{\s}{{\mathbb{S}}}

\newcommand{\I}{\mathcal{I}}
\newcommand{\study}{{\mathscr{S}}}
\newcommand{\Q}{{Q}}
\newcommand{\con}{\mathscr{C}}

\newcommand\re{\operatorname{Re}}
\newcommand\im{\operatorname{Im}}

\newcommand{\hslashslash}{%
  \raisebox{.9ex}{%
    \scalebox{.7}{%
      \rotatebox[origin=c]{18}{$-$}%
    }%
  }%
}
\newcommand{\fslash}{%
  {%
   \vphantom{f}%
   \ooalign{\kern.05em\smash{\hslashslash}\hidewidth\cr$f$\cr}%
   \kern.05em
  }%
}

\title{\bf Zeros of slice functions and polynomials over dual quaternions}

\author{Graziano Gentili, Caterina Stoppato, Tomaso Trinci\\
\\
\small Dipartimento di Matematica e Informatica ``U. Dini'', Universit\`a degli Studi di Firenze \\
\small Viale Morgagni 67/A, I-50134 Firenze, Italy\\
\small graziano.gentili@unifi.it, caterina.stoppato@unifi.it, tomaso.trinci90@gmail.com}

\date{  }

\begin{document}

\maketitle


\begin{abstract}
This work studies the zeros of slice functions over the algebra of dual quaternions and it comprises applications to the problem of factorizing motion polynomials. The class of slice functions over a real alternative *-algebra $A$ was defined by Ghiloni and Perotti in 2011, extending the class of slice regular functions introduced by Gentili and Struppa in 2006. Both classes strictly include the polynomials over $A$. We focus on the case when $A$ is the algebra of dual quaternions $\D\hh$. The specific properties of this algebra allow a full characterization of the zero sets, which is not available over general real alternative *-algebras. This characterization sheds some light on the study of motion polynomials over $\D\hh$, introduced by Heged{\"u}s, Schicho, and Schr{\"o}cker in 2013 for their relevance in mechanism science.
\end{abstract}


{\small
\tableofcontents
}


\thanks{\small \noindent{\bf Acknowledgements.} This work was partly supported by INdAM, through: GNSAGA; INdAM project ``Hypercomplex function theory and applications''. It was also partly supported by MIUR, through the projects: Finanziamento Premiale FOE 2014 ``Splines for accUrate NumeRics: adaptIve models for Simulation Environments''; PRIN 2017 ``Real and complex manifolds: topology, geometry and holomorphic dynamics''. The authors warmly thank the anonymous reviewers for their precious suggestions.}


\section{Introduction}\label{sec:introduction}

The concept of \emph{slice regular} function is one of the possible generalizations of the notion of complex analytic function to higher dimensional real algebras. It was introduced over quaternions in~\cite{cras, advances} and over other algebras in~\cite{clifford,rocky} (see also the related work~\cite{israel}). Then the work~\cite{perotti} extended the concept of slice regular function to all real alternative *-algebras and introduced the broader class of \emph{slice} functions. The theory of slice regular functions includes polynomials and it generalizes many classical results of complex analysis. In some cases the generalization is straightforward; in many others, it is more intricate and it reveals a richer environment. This is already evident when studying the zeros of slice regular functions: as proven in~\cite{gpsalgebra}, some form of discreteness of the zero set is always present, but the exact characterization of the zero set is algebra-specific because it depends on the nature of the zero divisors within the algebra.

Deepening the study of the zero sets of slice functions and slice regular functions by focusing on a specific algebra is exactly the purpose of this work. We focus on the algebra $\D\hh$ of dual quaternions, whose set of zero divisors is well understood, not only for its intrinsic interest but also for its applications to mechanism science, see~\cite{hegedus,lischarlerschrocker,lischichoschrocker1,lischichoschrocker2}. These articles introduced \emph{motion polynomials}, which correspond to rational rigid body motions in the Euclidean $3$-space, and studied their factorizations, which correspond to linkages producing the same motions.

The work is organized as follows. Section~\ref{sec:numbers} describes the construction of the algebra $\D\hh$, its properties, and its use to represent the group of proper rigid body transformations. The definitions and basic properties of slice functions and slice regular functions are recalled in Section~\ref{sec:functions}, along with the definition of motion polynomial. Section~\ref{sec:primal} constructs the \emph{primal part} of a slice function over $\D\hh$, extending a known construction over polynomials and providing a tool to fully exploit the peculiarities of the algebra of dual quaternions. The study of the zeros of slice functions over $\D\hh$ is first addressed in Section~\ref{sec:zeros}. It continues in Section~\ref{sec:products}, which focuses on the zeros of products of slice functions, and in Section~\ref{sec:factorization}, which is devoted to slice regular functions. Section~\ref{sec:applications} presents successful applications to the problem of factorizing motion polynomials.


\section{The algebras of quaternions and dual quaternions}\label{sec:numbers}
This section presents the algebra of quaternions $\hh$ and the algebra of dual quaternions $\D\hh$, along with relevant actions of their multiplicative groups.

\subsection{Quaternions}
Let $\hh$ denote the real algebra of quaternions. Recall that it is obtained by endowing $\rr^4$ with the multiplication operation defined on the standard basis $\{1,i,j,k\}$ by
\begin{align*}
i^2=j^2&=k^2=-1, \\
ij=-ji=k, \quad jk&=-kj=i, \quad ki=-ik=j,
\end{align*}
and extended by distributivity to all quaternions $q=x_0+x_1i+x_2j+x_3k$. We set 
\[ \re(q)=x_0, \quad \im(q)=x_1i+x_2j+x_3k, \quad |q|=\sqrt{x_0^2+x_1^2+x_2^2+x_3^2}.\]
$\re(q), \im(q)$ and $|q|$ are called the \emph{real part}, the \emph{imaginary part} and the \emph{modulus} of $q$, respectively. The quaternion
\[ q^c=x_0-x_1i-x_2j-x_3k=\re(q)-\im(q)\]
is called the \emph{conjugate} of $q$ and it satisfies
\[|q|=\sqrt{qq^c}=\sqrt{q^cq}.\]
The \emph{inverse} of any element $q\in\hh^*\vcentcolon= \hh\setminus\{0\}$ is given by
\[q^{-1}=\dfrac{q^c}{|q|^2}.\]
Two quaternions $p,q$ commute if, and only if, $\im(p)=c \im(q)$ for some $c \in \rr$. The commutative \emph{center} 
\[\mathcal{Z}(\hh) \vcentcolon= \{x \in \hh\,:\,xh=hx\ \forall\,h \in \hh\}\]
of $\hh$ coincides with $\rr$. Furthermore, for each pair of quaternions $p,q$, the standard scalar product between $p$ and $q$ equals $\frac12(pq^c+qp^c)$; when $p,q$ are orthogonal, i.e., $pq^c+qp^c=0$, we write $p \perp q$. For a more detailed discussion we refer the reader to~\cite{librospringer}.


\subsection{Dual quaternions}

Let $\D\hh$ denote the algebra of dual quaternions, which is an associative algebra over the real field $\rr$ defined as
\[ \D\hh \vcentcolon=\hh+\epsilon \hh, \]
with the following definitions for all $h=h_1+\epsilon h_2,h'=h'_1+\epsilon h'_2$:
\begin{description}
\item[Addition:] $h+h'=(h_1+h'_1)+\epsilon (h_2+h'_2)$;
\item[Multiplication:] $hh'=h_1h'_1+\epsilon(h_1h'_2+h_2h'_1)$.
\end{description}
For each $h=h_1+\epsilon h_2\in\D\hh$, we will refer to $h_1$ as the \emph{primal part} and $h_2$ as the \emph{dual part} of $h$.

\begin{remark}\label{rmk:centerdh}
The center $\mathcal{Z}({\D\hh})$ of $\D\hh$ coincides with the real (commutative) subalgebra $\D\rr\vcentcolon=\rr+\epsilon \rr$, whose elements are called \emph{dual numbers}.
\end{remark}

By construction, $\epsilon^2=0$. As a consequence, we can make the following remark.

\begin{remark}\label{rmk:zerodivisors}
A dual quaternion $h \in \D\hh^*=\D\hh \ \setminus \{0\}$ is a zero divisor if, and only if, $h=\epsilon h_2$, i.e., its primal part vanishes. In other words, the set of zero divisors in $\D\hh$ is $\epsilon \hh^*$ and the set of zero divisors in $\D\rr$ is $\epsilon \rr^*$. In particular, the product of any two zero divisors in $\D\hh$ equals zero.
\end{remark}

Since $\D\hh$ is associative, an element admits a multiplicative inverse if, and only if, it is neither zero nor a zero divisor.

\begin{remark}\label{rmk:inverse}
For every $h \in \D\hh\setminus\epsilon\hh$, the inverse of $h$ is
\[h^{-1}=h_1^{-1}-\epsilon h_1^{-1}h_2h_1^{-1}.\]
\end{remark}

The algebra $\D\hh$ is a *-algebra with a *-involution called \emph{conjugation}. For each $h \in \D\hh$, it is defined as follows.

\begin{description}
\item[Conjugation:] $h^c=h_1^c+\epsilon h_2^c$,
\end{description}
where $h_1^c, h_2^c$ are the standard conjugates of the quaternions $h_1$ and $h_2$.


\subsection{Imaginary units}\label{sec:imunit}

\begin{definition}
For every $h \in \D\hh\supseteq\hh$, the \emph{trace} of $h$ and the (squared) \emph{norm} of $h$ are defined as 
\[t(h)=h+h^c, \quad n(h)=hh^c.\]
\end{definition}

\begin{remark}
Trace and norm of any $h \in \D\hh$ are dual numbers, namely
\begin{align}
 n(h)&=h_1h_1^c+\epsilon(h_1h_2^c+h_2h_1^c) =|h_1|^2+2\epsilon \re(h_1h_2^c), \nonumber \\ 
 t(h)&=(h_1+h_1^c)+\epsilon (h_2+h_2^c)=2\re(h_1)+2\epsilon \re(h_2).  \nonumber
\end{align}
As a consequence, for all $h\in\hh$ the trace $t(h)=2\re(h)$ and the norm $n(h)=|h|^2$ are real numbers.
\end{remark}

Both for $A=\D\hh$ and for $A=\hh$, we call the elements of
\[\s_{A} \vcentcolon =\{h \in A\,:\,t(h)=0\text{, } n(h)=1\}\]
the \emph{imaginary units} of $A$. The set
\[\s_\hh=\{h \in \hh\,:\,\re(h)=0, |h|=1\}\]
is the unit $2$-sphere in the $3$-space of purely imaginary quaternions, while
\[\s_{\D\hh}=\{h_1+\epsilon h_2 \in \D\hh\,:\,h_1\in\s_\hh, h_2\in\im(\hh), h_1\perp h_2\}\]
can be seen as the total space of the tangent bundle over $\s_\hh$. Indeed, any element $J \in \s_{\D\hh}$ splits into a primal part $J_1\in\s_\hh$ and a dual part $J_2$ which equals $mI$ for some $m \in \rr$ and some $I\in\s_\hh$ orthogonal to $J_1$ and we may think of the set
\[T_{J_1}\vcentcolon=\{J_1+\epsilon m I: m \in \rr, I\in\s_\hh, I\perp J_1\}\]
as the tangent plane to $\s_\hh$ at $J_1$.

\begin{remark}
It is well-known that $\s_\hh=\{J\in\hh\,:\,J^2=-1\}$. As a consequence of this fact and of the equality $h^2=h_1^2+\epsilon(h_1h_2+h_2h_1)$, we also have that
\[\s_{\D\hh}=\{J\in\D\hh\,:\,J^2=-1\}\,.\]
\end{remark}

Both for $A=\D\hh$ and for $A=\hh$, the set of elements $x$ that can be uniquely expressed as $x=\alpha+\beta J$ for some $\alpha,\beta \in \rr$, $\beta>0$ and $J \in \s_A$ is the \emph{quadratic cone}
\[\Q_A \vcentcolon = \rr \cup \{h \in A\,:\,t(h),n(h) \in \rr  \text{ and } 4n(h)>t(h)^2 \}.\]
Indeed, it has been proven in~\cite{perotti} that
\[ \Q_A  = \bigcup_{J \in \s_A} \cc_J,\]
where $\cc_J=\rr+\rr J$ is the unitary *-subalgebra generated by $1$ any $J$. Clearly, $\cc_J$ is isomorphic to the complex field and $\cc_I \cap \cc_J = \rr$ for every $I,J \in \s_A$ with $I\neq \pm J$. Moreover, $Q_A$ splits as the disjoint union
\[Q_A\ =\bigcup_{\alpha\in\rr,\beta\geq0} (\alpha + \beta \s_A)\,.\]
In other words, each $x=\alpha+\beta J \in \Q_A$ lies in a unique 
\[\s_x \vcentcolon= \alpha + \beta \s_A\,,\]
which is obtained by real translation and dilation from $\s_A$ when $\beta>0$ and is a singleton $\{\alpha\}$ when $\beta=0$.

It has been proven in~\cite{perotti} that
\[\Q_\hh=\hh\,.\]
Moreover,
\[\Q_{\D\hh}=\rr \cup \{h_1+\epsilon h_2 \in \D\hh\,:\,h_1\in\hh\setminus\rr, h_2\in\im(\hh), h_1\perp h_2\}\,.\]

The sets $\Q_{\D\hh}$ and $\s_{\D\hh}$ have interesting expressions in Cartesian coordinates. Let us consider the standard basis $\{1,i,j,k\}$ of $\hh$. Any dual quaternion can be expressed by
\begin{equation*}
 h=\underbrace{r_0+r_1 i + r_2 j+ r_3 k}_{h_1} + \epsilon (\underbrace{r_4 +  r_5 i+  r_6 j+  r_7 k}_{h_2}),
\end{equation*}
with $r_s \in \rr$ for all $s \in [0,\dots,7]$. The condition $h_1 \perp h_2$ is equivalent to the equation of the \emph{Study quadric}
\begin{equation}
\study^7 : r_0r_4+r_1r_5+r_2r_6+r_3r_7=0\,,
\end{equation}
which will play an important role throughout the paper. The condition $h_2\in\im(\hh)$ is equivalent to the equation $r_4=0$. Finally, $h_1\in\hh\setminus\rr$ if, and only if, $r_1^2+r_2^2+r_3^2 \neq 0$. To sum up, $\Q_{\D\hh}$ is determined by the conditions
\[r_4=0, \quad r_1r_5+r_2r_6+r_3r_7=0, \quad  r_1^2+r_2^2+r_3^2 \neq 0\]
and it has dimension $6$. In particular, the elements of $\s_{\D\hh}$ are determined by the equations
\[ r_0=r_4=0, \quad r_1r_5+r_2r_6+r_3r_7=0, \quad  r_1^2+r_2^2+r_3^2 =1.\]
For future reference, we point out that the intersection between $\study^7$ and the subspace $r_4=0$ is the union
\[(\rr+\epsilon\im(\hh))\cup\Q_{\D\hh}\,,\]
where the first operand is a $4$-space that intersects $\Q_{\D\hh}$ along the axis $\rr$.


\subsection{Conjugacy}

Let us consider the multiplicative subgroup
\[\D\hh^\times=\D\hh \setminus \epsilon\hh\]
of the algebra $\D\hh$. This group acts on $\D\hh$ as follows.

\begin{definition}
We define
\begin{equation}\label{eq:action}
\con : \D\hh^\times \times \D\hh \longrightarrow \D\hh, \qquad (h,l)\longmapsto \con(h,l)\vcentcolon= h^{-1}lh.
\end{equation}
We will also use the notation $\con_h(l)\vcentcolon =\con(h,l)$.
\end{definition}

Thanks to the associativity of $\D\hh$, it follows immediately that $\con$ is an action. Le us denote by $[a,b]:=ab-ba$ the commutator of any $a,b\in\D\hh$. As a consequence of Remark~\ref{rmk:inverse},
\begin{equation}\label{eq:action2}
\con(h,l)
=h_1^{-1}l_1h_1+\epsilon\left(h_1^{-1}l_2h_1+ [h_1^{-1}l_1h_1, h_1^{-1}h_2]\right).
\end{equation}

\begin{proposition}
For each $h\in \D\hh^\times$, the map $\con_h$ maps $\s_{\D\hh}$ bijectively into itself. As a consequence, for all $x\in Q_{\D\hh}$, the map $\con_h$ maps $\s_x$ bijectively into itself.
\end{proposition}

\begin{proof}
Let us consider an arbitrary element $h\in \D\hh^\times$ and verify that $\con_h(J) \in \s_{\D\hh}$ for every $J \in \s_{\D\hh}$. Recalling the definition of $\s_{\D\hh}$, it is sufficient to evaluate the norm and the trace of the element $h^{-1}Jh$:
\[n(h^{-1}Jh)=(h^{-1}Jh)(h^{-1}Jh)^c=h^{-1}Jhh^c(-J)(h^{-1})^c=n(h)^{-1}{h^c}n(h)n(h)^{-1}h=1, \]
\[t(h^{-1}Jh)=h^{-1}Jh+(h^{-1}Jh)^c=h^{-1}Jh+h^c(-J)(h^{-1})^c=n(h)^{-1}{h^c}Jh+h^c(-J)n(h)^{-1}h=0.\]
Moreover, we observe that $(\con_h)_{|_{\s_{\D\hh}}}:\s_{\D\hh}\to\s_{\D\hh}$ is a bijection with inverse $(\con_{h^{-1}})_{|_{\s_{\D\hh}}}$.
Finally, for all $x=\alpha+\beta J \in \Q_{\D\hh}$ we have that $(\con_h)_{|_{\s_x}}$ is a transformation of $\s_x=\alpha+\beta\s_{\D\hh}$ as a consequence of the equality
\begin{equation}\label{eq:actiononsx}
\con_h(\alpha+\beta I)=h^{-1}(\alpha+\beta I)h=\alpha+\beta h^{-1}Ih =\alpha+\beta \con_h(I)\,,
\end{equation}
valid for all $I\in\s_{\D\hh}$.
\end{proof}

The previous proposition implies that the action $\con$ on $\D\hh$ is not transitive. However, we can prove transitivity on a single $\s_x$.

\begin{proposition}
Fix $x \in Q_{\D\hh}$. The action of $\D\hh^\times$ on $\s_x$ is transitive. If we quotient by $\D\rr^\times =\D\rr\ \setminus \epsilon \rr$ then the action of $\D\hh^\times/\D\rr^\times$ on $\s_x$ is faithful, but not free.
\end{proposition}

\begin{proof}
Thanks to equality~\eqref{eq:actiononsx}, it suffices to prove the theses for $\s_{\D\hh}$ (i.e., for $x\in\s_{\D\hh}$).

Let us fix $J\in\s_{\D\hh}$ and prove that
\[h\mapsto \con(h,J)=h_1^{-1}J_1h_1+\epsilon\left(h_1^{-1}J_2h_1 + [h_1^{-1}J_1h_1, h_1^{-1}h_2]\right).\]
is a surjective map $\D\hh^\times\to\s_{\D\hh}$. For any $K_1\in\s_{\hh}$, the equality
\[h_1^{-1}J_1h_1=K_1\]
is fulfilled for $h_1=J_1+K_1$ when $K_1\neq-J_1$; it is fulfilled for any $h_1\in\s_\hh$ with $h_1\perp J_1$ when $K_1=-J_1$. Now we want to prove that, for $K_1\in\s_{\hh}$ fixed and for the aforementioned $h_1\in\hh$, the map 
\[h_2\mapsto h_1^{-1}J_2h_1 + [K_1, h_1^{-1}h_2]\]
from $\hh$ to the tangent plane at $K_1$ to $\s_{\hh}$ in $\hh$ is surjective. To this end, it suffices to prove that the real linear map 
\[h_2\mapsto[K_1, h_1^{-1}h_2]\]
has rank $2$. Let us choose $L_1\in\s_\hh$ with $L_1\perp K_1$ and set $M_1\vcentcolon =K_1L_1=\frac12[K_1,L_1]$. Then the last displayed linear map transforms the vectors $h_1L_1,h_1M_1$ into the linearly independent vectors $[K_1,L_1]=2M_1, [K_1,M_1]=-2L_1$, as desired.
We have thus proven that the action of $\D\hh^\times$ on $\s_{\D\hh}$ is transitive.

This action is clearly not faithful, but it reduces to a faithful action if we quotient $\D\hh^\times$ with its center, which is the maximal normal subgroup of $\D\hh^\times$ included in all stabilizers of the action. As a consequence of the fact that the center of $\D\hh$ is $\D\rr$, we conclude that the center of $\D\hh^\times$ is $\D\rr^\times=\D\rr\setminus\epsilon\rr$.

Finally, we can see that the quotient action is not free by considering the stabilizer of $i$. Indeed, since $ih=hi$ is equivalent to $h\in\D\cc\vcentcolon=\cc+\epsilon\cc$, we conclude that the stabilizer of $i$ equals $\D\cc^\times/\D\rr^\times$, where $\D\cc^\times=\D\cc\setminus\epsilon\cc$.
\end{proof}

The next definition will be useful in the sequel. It extends $\con$ to $\D\hh\times\D\hh$, although the first factor in the Cartesian product is no longer a group.

\begin{definition}
We define
\begin{equation*}
\label{def:actiongen}
\con :\D\hh \times \D\hh \longrightarrow \D\hh, \qquad (h,l)\longmapsto \con(h,l)\vcentcolon=
\begin{cases}
h^{-1}lh & \text{if\ }h\in\D\hh^\times\\
h_2^{-1}lh_2 & \text{if\ }h\in\epsilon\hh^*\\
l & \text{if\ }h=0
\end{cases}
\end{equation*}
 and $\con_h(l):=\con(h,l)$.
\end{definition}

This peculiar extension is justified by the next remarks.

\begin{remark}\label{rmk:conproprerty}
For all $x \in Q_{\D\hh}$, the extended $\con$ maps $\D\hh\times\s_x$ surjectively into $\s_x$.  Moreover, for all $h \in \D\hh$, the equalities $\re(\con_h(x))=\re(x), \im(\con_h(x))=\con_h(\im(x))$ hold.
\end{remark}

\begin{remark}\label{rmk:conproprerty2}
For all $h,l \in \D\hh$, it holds 
\[h\con_h(l)=lh\,.\]
This equality is obvious if $h$ is invertible or $h=0$. If $h\in\epsilon\hh^*$, it follows by direct computation: $h\con_h(l)=\epsilon h_2 (h_2^{-1}l h_2) = l(\epsilon h_2)=lh$. In the last case, we actually have $h\tilde l=lh$ for all $\tilde l\in T_{\con_h(l_1)}$.
\end{remark}

\begin{remark}
Let us fix $l\in\D\hh$. Then $h_2^{-1}lh_2$ is a limit point of $h^{-1}lh$ as $h_1\to0$. Indeed,
\[\lim_{\rr\ni t\to0}\con(th_2+\epsilon h_2,l) = \lim_{\rr\ni t\to0}\left((th_2)^{-1}l(th_2)+\epsilon [(th_2)^{-1}l_1(th_2),(th_2)^{-1}h_2]\right)=h_2^{-1}lh_2\,.\]
Moreover, $l$ is a limit point of $h^{-1}lh$ as $h\to0$ because $\lim_{\rr\ni t\to0}\con(t,l)=\lim_{\rr\ni t\to0}l=l$.
\end{remark}


\subsection{Euclidean displacement}

It is well-known, see~\cite{daniilidis, selig}, that dual quaternions can be used to represent the group of proper rigid body transformations $SE(3)$. Let us outline this  characterization, starting with a few necessary tools.

First of all, $\rr^3$ can be identified with $1+\epsilon\im(\hh)$ by seeing the vector $(x_1,x_2,x_3)$ as the dual quaternion $1+\epsilon(x_1i+x_2j+x_3k)$. Proper rigid body transformations can be obtained by means of an appropriate action on $1+\epsilon\im(\hh)$ of the following subgroup of the multiplicative group $\D\hh^\times$:
\[G:=\{h\in\D\hh\,:\,n(h)\in\rr^*\}=\{h_1+\epsilon h_2\,:\,h_1,h_2\in\hh,h_1\neq0,t(h_1h_2^c)=0\}=\D\hh^\times\cap\study^7\,.\]
In order to define this action, let us introduce a new *-involution on $\D\hh$ (different from $h\mapsto h^c$):
\begin{description}
\item[Alternate *-involution:] $\tilde h=h_1^c-\epsilon h_2^c$.
\end{description}
Now, to each $h\in G$, we can associate the transformation
\begin{equation}\label{eq:rigidbodytrasnformation}
1+\epsilon x \longmapsto \frac{h\,(1+\epsilon x)\,\tilde h}{n(h)} =1+\epsilon(h_1xh_1^{-1}+2h_2h_1^{-1})\,.
\end{equation}
If we restrict to the case when $t(h)\in\rr$, i.e., $h_2\in\im(\hh)$, we find two special cases.
\begin{description}
\item[Translations.] If $h\in \rr+\epsilon\im(\hh)$, then the previous transformation is a pure translation of $\rr^3$ with translation vector $2h_2h_1^{-1}$. In particular, if $h_2=0$ then the transformation is the identity.
\item[Rotations.] If $h\in \Q_{\D\hh}\setminus\rr$, then the previous transformation is a pure rotation of $\rr^3$. The rotation angle $\theta$ is determined by the equality $\cos(\frac\theta2)=\frac{\re(h_1)}{|h_1|}$. The rotation axis has Pl\"ucker coordinates $\big(\frac{\im(h_1)}{|\im(h_1)|},\frac{h_2}{|\im(h_1)|}\big)$. In particular, if $h=h_1\in\hh\setminus\rr$, then the rotation axis is the line through the origin parallel to $\im(h_1)$. This corresponds to the double covering of $SO(3)$ by means of group of unitary quaternions described in~\cite[Theorem~3.17]{GHS}.
\end{description}
More details can be found in~\cite[Section 4]{daniilidis}.

As proven in~\cite[Chapter 9]{selig}, the isomorphism
\[SE(3)\simeq G/\rr^*\]
holds and the group of dual quaternions $h$ with $n(h)=1$ is a double covering of $SE(3)$.


\section{The algebra of slice functions}\label{sec:functions}

In this section, we let the symbol $A$ refer to $\hh$ or $\D\hh$ indistinctly. We overview some material from~\cite{perotti,gpsalgebra} concerning the theory of slice functions over the *-algebra $A$.

The *-algebra $A$ can be seen as $2n$-dimensional vector space over $\rr$ for $n=2$ or $n=4$. The left multiplication by a element $J \in \s_A$ induces a complex structure on $A$, thus there exist vectors $e_1^J, \dots ,e_{n-1}^J \in A$ such that the set $\{1,J,e^J_1, Je^J_1,\dots,e^J_{n-1},Je^J_{n-1}\}$ is a real vector basis, called a \emph{splitting basis} of $A$ associated to $J$. We consider on $A$ the natural Euclidean topology and differential structure. The relative topology on each $\cc_J$ with $J \in \s_A$ clearly agrees with the topology determined by the natural identification between $\cc_J$ and $\cc$, through the *-isomorphism
\[\phi_J:\cc \longrightarrow \cc_J, \quad \alpha+\beta i \longmapsto \alpha+ \beta J.\]
Given a subset $D$ of $\cc$, its circularization $\Omega_D$ is defined as the following subset of $\Q_A$:
\[ \Omega_D \vcentcolon=\{x\in \Q_A\ : \ \exists\, \alpha, \beta \in \rr, \exists\, J \in \s_A \text{ s.t. } x=\alpha+\beta J, \alpha + \beta i \in D \}. \]
A subset of $\Q_A$ is termed \emph{circular} if it equals $\Omega_D$ for some $D \subset \cc$. For instance, given $x=\alpha+\beta J \in \Q_A$ the smallest circular set including $x$ is
\[\Omega_{\{x\}} = \s_x = \alpha + \beta \s_A\,.\]
From now on, we assume $D$ to be invariant under complex conjugation $z=\alpha+\beta i \mapsto \alpha - \beta i$. As a consequence, for each $J \in \s_A$ the ``slice'' $\Omega_J \vcentcolon=\Omega \cap \cc_J$ is equivalent to $D$ under the natural identification between $\cc_J$ and $\cc$.

The class of $A$-valued functions we consider was defined in \cite{perotti} by means of the complexified algebra $A_{\cc} =A \otimes_{\rr} \cc =\{x+\iota y \ :  \ x,y \in A, \iota ^2=-1\}$ of $A$, endowed with the following product:
\[(x+\iota y)(x' +\iota y')=xx' - yy' + \iota (xy' +yx').\]
In this setting, $A$ can be found as the subalgebra $A+\iota 0$ of $A_\cc$. An isomorphic copy of $\cc$ can be found as the subalgebra $\rr_\cc=\rr \otimes_{\rr} \cc=\rr + \iota \rr$ of $A_\cc$. Henceforth, we identify $\cc$ (which includes $D$) with $\rr_\cc$.

\begin{remark}
We notice that if $A=\hh$, then $\mathcal{Z}(\hh_\cc)=\rr_\cc$. Similarly, if $A=\D\hh$ then $\mathcal{Z}(\D\hh_\cc)=\D\rr_\cc$.
\end{remark}

In addition to the complex conjugation
\[ \overline{x+\iota y}=x-\iota y, \]
the *-involution on $A$ induces a *-involution on $A_{\cc}$, namely
\[(x+\iota y)^c=x^c+\iota y^c.\]
For all $J \in \s_A$, we can extend the previously defined map $\phi_J: \cc \rightarrow \cc_J$ to
\[\phi_J:A_{\cc} \longrightarrow A, \ \ x+\iota y \longmapsto x+Jy.\]
Let $D$ be a subset of $\cc$ and consider a function
\[F=F_1+\iota F_2: D \longrightarrow A_{\cc}\]
with $A$-valued components $F_1$ and $F_2$. The function $F$ is called a \emph{stem function} on $D$ if $F(\overline{z})=\overline{F(z)}$ for every $z \in D$ or, equivalently, if $F_1(\overline{z})=F_1(z)$ and $F_2(\overline{z})=-F_2(z)$ for every $z \in D$.

\begin{definition}
A function $f:\Omega_D \longrightarrow A$ is called a \emph{(left) slice function} if there exists a stem function $F : D \longrightarrow A_{\cc}$ such that the diagram
\begin{equation}
\begin{CD}
D @>F> >A_\cc\\ 
@V V \phi_J V 
@V V \phi_J V\\
\Omega_D@>f> >A 
\end{CD}
\end{equation}
commutes for each $J\in \s_A$. In this situation, we say that $f$ is induced by $F$ and we write $f = \I(F)$. If $F$ is $\rr_{\cc}$-valued, then we say that the slice function $f$ is \emph{slice preserving}. 
\end{definition}


The algebraic structure of slice functions is described by the following proposition. A detailed proof can be found in \cite{perotti}.

\begin{proposition}
The stem functions $D\longrightarrow A_{\cc}$ form a *-algebra over $\rr$ with pointwise addition $(F +G)(z) = F (z)+G(z)$, multiplication $(F G)(z) = F (z)G(z)$ and conjugation $F^c(z) = F(z)^c$. This *-algebra is associative and its center includes the subset of $\rr_{\cc}$-valued stem functions. Let $\Omega = \Omega_D$ and consider the mapping
\[ \I:\{ \text{\emph{stem functions on}} \ D \} \longrightarrow \{\text{\emph{slice functions on}} \ \Omega \} =\vcentcolon \mathcal{S}(\Omega) \]
Besides the pointwise addition $(f, g) \mapsto f + g$, there exist unique operations of multiplication $(f,g) \mapsto f \cdot g$ and conjugation $f \mapsto f^c$ on $\mathcal{S}(\Omega)$ such that the mapping $\I$ is a *-algebra isomorphism. The *-algebra $\mathcal{S}(\Omega)$ is associative and its center $\mathcal{Z}(\mathcal{S}(\Omega))$ includes the *-subalgebra $S_{\rr}(\Omega)$ of slice preserving functions.
\end{proposition}

The next result describes the centers of the *-algebras of quaternionic and dual quaternionic slice functions.

\begin{proposition}
If $A=\hh$ and $\Omega=\Omega_D \subseteq \hh$, then $\mathcal{Z}(\mathcal{S}(\Omega))=S_{\rr}(\Omega)$. If $A=\D\hh$ and $\Omega=\Omega_D \subseteq \Q_{\D\hh}$, then $\mathcal{Z}(\mathcal{S}(\Omega))$ is the *-subalgebra $S_{\D\rr}(\Omega)$ of slice functions induced by $\D\rr_\cc$-valued stem functions, which properly includes the *-subalgebra  $S_{\rr}(\Omega)$ of slice preserving functions.
\end{proposition}

\begin{proof}
Let us prove the statement for $A=\D\hh$: the case $A=\hh$ is well-known and can be proven using the same technique.
Thanks to the previous proposition, it suffices to prove that the center of the *-algebra of stem functions $D\to \D\hh_\cc$ is the *-subalgebra of stem functions $D\to \D\rr_\cc$.

If $F:D\to \D\rr_\cc$ is a stem function, then it commutes with any stem function $G:D\to \D\hh_\cc$ by Remark~\ref{rmk:centerdh}.

Conversely, suppose $F:D\to \D\hh_\cc$ to commute with all stem functions $G:D\to \D\hh_\cc$. In particular, $F$ commutes with all $G=G_1+\iota G_2$ with $G_1\equiv a$ for $a\in\D\hh$ and $G_2\equiv0$. If we fix $z\in D$, it follows that
\[F(z)a=F(z)G(z)=G(z)F(z)=aF(z)\,.\]
By Remark~\ref{rmk:centerdh}, $F(z)\in\D\rr_\cc$, as desired.
\end{proof}

The product $f \cdot g$ of two functions $f, g \in \mathcal{S}(\Omega)$ is called \emph{slice product} of $f$ and $g$. If $f$ belongs to the center of $\mathcal{S}(\Omega)$ then, by direct inspection $(f\cdot g)(x)=f(x)g(x)$ for all $x\in\Omega$. In general to compute $f \cdot g$, one needs instead to compute $FG$ and then $f\cdot g = \I(FG)$. Similarly, for the \emph{slice conjugate} $f^c$ of $f = \I(F)$ we compute $f^c = \I(F^c)$. The \emph{normal function} of $f$ in $\mathcal{S}(\Omega)$ is defined as
\[N(f)=f \cdot f^c =\I(FF^c).\]
Formulae to express the operations on slice functions without computing the corresponding stem functions can be given by means of two further operations. To each $f \in \mathcal{S}(\Omega)$, we associate a function $f_s^\circ : \Omega \longrightarrow A$, called \emph{spherical value} of $f$, and a function $f_s' : \Omega \setminus \rr \longrightarrow A$, called \emph{spherical derivative} of $f$, by setting
\begin{align}
\label{eq:der}
f^\circ_s(x) & \vcentcolon = \frac{1}{2}(f(x)+f(x^c)), \\
\label{eq:val}
f'_s(x) & \vcentcolon = \frac{1}{2} \im(x)^{-1}(f(x)-f(x^c)).
\end{align}
Spherical value and spherical derivative are slice functions, too: if $f=\I(F_1+\iota F_2)$ then $f^\circ_s=\I(F_1)$ and $f'_s=\I(\widetilde F_2)$ with $\widetilde F_2(\alpha+\iota\beta)\vcentcolon=\beta^{-1}F_2(\alpha+\iota\beta)$. Clearly,
\begin{equation}\label{eqn:repr}
f(x)=f^\circ_s(x)+\im(x)f'_s(x)\,.
\end{equation}
What is less obvious, but a consequence of the definition of slice functions, is the fact that $f^\circ_s,f'_s$ are constant on each $\s_x \subseteq \Omega$. As a consequence, $f\in\mathcal{S}(\Omega)$ is uniquely determined by its restriction $f_{|\Omega_J}$ to any slice $\Omega_J$ (with $J\in\s_A$) of its domain.

We are now ready to state the aforementioned formulae for the operations on slice functions (see~\cite{gpsalgebra}): for all $x \in \Omega \setminus \rr$,
\begin{align}
f^c&=(f^c)^\circ_s+\im \ (f^c)'_s; \label{eqn:con} \\
f\cdot g&=\underbrace{f^\circ_sg^\circ_s+ \im^2 \ f'_sg'_s}_{(f \cdot g)^\circ_s} + \im \ \underbrace{(f^\circ_sg'_s+f'_sg^\circ_s)}_{(f \cdot g)'_s}; \label{eqn:prod}\\
N(f)(x)&=\underbrace{n(f^\circ_s(x))+\im(x)^2 n(f'_s(x))}_{N(f)^\circ_s(x)}+ \im(x)\underbrace{t(f^\circ_s(x)f'_s(x)^c)}_{N(f)'_s(x)} \label{eqn:normal}.
\end{align}

\begin{remark}
Consider $h,h' \in A$ then, by direct computation, $n(h)=n(h^c)$ and $t(hh')=t(h'h)$. As a consequence of formula~\eqref{eqn:normal}, for all $f \in \mathcal{S}(\Omega)$ we have that $N(f)=N(f^c)$. Moreover, since $n$ and $t$ take values in the center of $A$, it follows that $N(f) \in \mathcal{Z}(\mathcal{S}(\Omega))$.
\end{remark}

\begin{definition}
A function $f\in\mathcal{S}(\Omega)$ is \emph{tame} if $N(f)=N(f^c)$ is slice preserving.
\end{definition}

For $A=\hh$, all slice functions are tame. For $A=\D\hh$, the tame elements of $\mathcal{S}(\Omega)$ form a proper subset of $\mathcal{S}(\Omega)$, which is closed under multiplication by~\cite[Remark 2.7]{gpsalgebra}.

Within the class of slice functions, we consider a special subclass having nice properties that recall those of holomorphic functions of a complex variable. Suppose $\Omega=\Omega_D$ to be open in $\Q_A$, then for any $J \in \s_A$, $\Omega \cap \cc_J$ is open in the relative topology of $\cc_J$; therefore, $D$ itself is open. We let $\mathcal{S}^0(\Omega)$ and $\mathcal{S}^1(\Omega)$ denote the real vector spaces of slice functions on $\Omega$ induced by continuous stem functions and by stem functions of class $C^1$, respectively. Now consider a function $f=\I(F) \in \mathcal{S}^1(\Omega)$; for $z=\alpha+ \iota \beta$, set
\begin{equation*}
\frac{\partial F}{\partial z}: D \longrightarrow A_{\cc}, \quad \frac{\partial F}{\partial z}\vcentcolon = \frac12\left(\frac{\partial F}{\partial \alpha} - \iota \frac{\partial F}{\partial \beta}\right).
\end{equation*}
\begin{equation*}
\frac{\partial F}{\partial \overline{z}}: D \longrightarrow A_{\cc}, \quad \frac{\partial F}{\partial \overline{z}}\vcentcolon = \frac12\left(\frac{\partial F}{\partial \alpha} + \iota \frac{\partial F}{\partial \beta}\right).
\end{equation*}
Both $\frac{\partial F}{\partial z}$ and $\frac{\partial F}{\partial \overline{z}}$ are still stem functions. They induce the slice functions
\[
\frac{\partial f}{\partial x^c} \vcentcolon = \mathcal{I}\bigg(\frac{\partial F}{\partial \overline{z}}\bigg), \qquad \frac{\partial f}{\partial x} \vcentcolon = \mathcal{I}\bigg(\frac{\partial F}{\partial z}\bigg).
\]

\begin{definition}
Let $\Omega$ be open in $\Q_A$. A slice function $f \in \mathcal{S}^1(\Omega)$ is called \emph{slice regular} if $\frac{\partial f}{\partial x^c}=0$ in $\Omega$. We denote by $\mathcal{SR}(\Omega)$ the real vector space of slice regular functions on $\Omega$. For each $f\in\mathcal{SR}(\Omega)$, the slice regular function $\frac{\partial f}{\partial x}$ is called the \emph{slice derivative} (or complex derivative) of $f$.
\end{definition}

The following lemma explains the connection between slice regularity and complex holomorphy. The proof of this result can be found in \cite{perotti}.

\begin{lemma}\label{lemma:splitting}
Suppose $\Omega=\Omega_D$ to be open in $\Q_A$. Let $J \in  \s_A$ and let $\{1,J,e^J_1, Je^J_1,\dots,e^J_{n-1},Je^J_{n-1}\}$ be an associated splitting basis of $A$. For $f \in \mathcal{S}^1(\Omega)$, let $f_0,\dots,f_{n-1}:\Omega_J \longrightarrow \cc_J$ be the $C^1$ functions such that $f_{|\Omega_J}=\sum_{l=0}^{n-1} f_le^J_l$, where $e^J_0\vcentcolon =1$. Then $f$ is slice regular if, and only if, for each $l \in \{0,\dots,n-1\}$,  $f_l$ is holomorphic from $\Omega_J$ to $\cc_J$, both equipped with the complex structure associated to left multiplication by $J$.
\end{lemma}

It has been proven in \cite{perotti} that slice regularity is closed under addition, slice multiplication and slice conjugation. Thus, $\mathcal{SR}(\Omega)$ is a *-subalgebra of $\mathcal{S}(\Omega)$.

The next result, also from \cite{perotti}, provides a relevant class of examples of slice regular functions over $\D\hh$, that will be particularly useful throughout the paper.

\begin{proposition}
\label{prop:algpol}
Let $\D\hh[t]$ denote the *-algebra of polynomials $\sum_{n=0}^dt^na_n$ over dual quaternions, with the standard operations
\begin{align*}
&\sum t^na_n+\sum t^nb_n=\sum t^n(a_n+b_n),\\
&\sum t^na_n\cdot\sum t^nb_n=\sum t^n\sum_{\ell=0}^na_\ell b_{n-\ell},\\
&\left(\sum t^na_n\right)^c=\sum t^n a_n^c.
\end{align*}
Let denote $\hh[t]$ the *-subalgebra of quaternionic polynomials. Mapping each polynomial $P(t)$ into the function $P_{|_{\Q_{\D\hh}}}:\Q_{\D\hh}\to\D\hh$ defines an injective *-algebra homomorphism $\D\hh[t]\to\mathcal{SR}(\Q_{\D\hh})$. Moreover, the inclusion $\hh[t]\longrightarrow \mathcal{SR}(\hh)$ is an injective *-algebra homomorphism.
\end{proposition}

If we take into account that $\D\rr$ is both the center of $\D\hh$ and the subspace of the points of $\D\hh$ preserved by conjugation, a direct inspection in the previous definition allows the following remark.

\begin{remark}\label{rmk:evaluation}
Let $P(t),Q(t)\in\D\hh[t]$, let $R(t):=P(t)\cdot Q(t)$ and let us evaluate these three polynomials at a point $h\in\D\hh$. While $(P+Q)(h)=P(h)+Q(h)$, the equalities
\begin{align*}
&R(h)=P(h)Q(h)\,,\\
&P^c(h)=P(h)^c.
\end{align*}
are only guaranteed if $h$ belongs to $\D\rr$ or if $P(t)$ belongs to $\D\rr[t]$.
\end{remark}

As explained in~\cite{hegedus}, polynomials over dual quaternions are particularly relevant for applications in kinematics. Let us sketch the relevant construction. We point out out that we are considering polynomials with coefficients on the right-hand side, as opposed to the convention adopted in~\cite{hegedus}. The polynomials $\sum_{n=0}^dt^na_n$ and $\sum_{n=0}^da_nt^n$ coincide when evaluated at real points, but not when evaluated at other points of $\D\hh$.

\begin{definition}
Consider a polynomial $P(t)=\sum_{n=0}^dt^na_n\in\D\hh[t]$ having degree $d$. The polynomial $N(P)(t):=P(t)\cdot P^c(t)\in\D\rr[t]$ is called the \emph{norm} of $P(t)$. If $N(P)(t)$ belongs to $\rr[t]$ and the leading coefficient $a_d$ of $P(t)$ belongs to $\D\hh^\times$, then $P(t)$ is called a \emph{motion polynomial}.
\end{definition}

If $P(t)$ is a motion polynomial, then we can make the following observations:
\begin{itemize}
\item for each $t_0\in\rr$, the value $h=P(t_0)$ has the property $n(h)=N(P)(t_0)\in\rr$, whence it belongs to the Study quadric $\study^7$;
\item the leading coefficient of $N(P)(t)$ is the $2d$-th coefficient, namely $n(a_d)\in\rr^*$; for any $t_0\in\rr$ that is not a root of $N(P)(t)$, the value $h=P(t_0)$ is an element of $G=\D\hh^\times\cap\study^7$.
\end{itemize}
As a consequence, for any $t_0\in\rr$ that is not a root of $N(P)(t)$ it is possible to consider the proper rigid body transformation~\eqref{eq:rigidbodytrasnformation} with $h=P(t_0)$:
\begin{equation}\label{eq:actionpoly}
1+\epsilon x\longmapsto \frac{P(t_0)\,(1+\epsilon x)\,\widetilde{P(t_0)}}{N(P)(t_0)}.
\end{equation}
If we fix a point, say $1+\epsilon x_0$, then its \emph{trajectory}
\begin{equation}\label{eq:trajectory}
t\longmapsto \frac{P(t)\,(1+\epsilon x_0)\,\widetilde{P(t)}}{N(P)(t)}
\end{equation}
will be a rational curve.

The following remark relates motion polynomials to tame functions and their properties.

\begin{remark}\label{rmk:motionvstame}
Let us consider a slice regular function $f$ defined as $f:=P_{|_{\Q_{\D\hh}}}$ for some $P(t)=\sum_{n=0}^dt^na_n\in\D\hh[t]$ having degree $d$. Then its normal function $N(f)$ coincides with $N(P)_{|_{\Q_{\D\hh}}}$. We can draw the following consequences.
\begin{itemize}
\item $P(t)$ is a motion polynomial if, and only if, $f$ is tame and $a_d\in\D\hh^\times$.
\item If $P(t)$ is a motion polynomial then $f$ cannot be a zero divisor in $\mathcal{S}^0(\Q_{\D\hh})$ by~\cite[Proposition 5.18]{gpsalgebra} and $P(t)$ cannot be a zero divisor in $\D\hh[t]$.
\end{itemize}
\end{remark}


\section{Primal part function}\label{sec:primal}

In this section we associate to each $\D\hh$-valued slice regular function $f$ an $\hh$-valued slice regular function, called the \emph{primal part} of $f$. This notion extends the analogous notion defined for polynomials in \cite{hegedus}. We begin with some preliminary definitions and results.

\begin{definition}
Let $\pi :\D\hh \longrightarrow \hh$ be the function that maps a dual quaternion into its primal part, i.e., $\pi(h)=\pi(h_1+\epsilon h_2)=h_1$.
\end{definition}

\begin{lemma}\label{lemma:homomorphism}
The function $\pi: \D\hh \longrightarrow \hh$ is a surjective real *-algebra homomorphism. Its extension $\pi_\cc:\D\hh_\cc \longrightarrow \hh_\cc$, defined as $\pi_\cc(x+\iota y)=\pi(x)+\iota\pi(y)$, is a surjective complex *-algebra homomorphism. 
\end{lemma}
\begin{proof}
By construction, $\pi$ is a surjective $\rr$-linear map. Moreover, for all $h,l \in \D\hh$, the following equalities hold:
\begin{align*}
\pi (hl) &=\pi (h_1l_1+\epsilon(h_1l_2+h_2l_1))=h_1l_1=\pi (h) \pi(l), \\
\pi (h^c)&= \pi (h_1^c+\epsilon h_2^c) = h_1^c= \pi (h)^c, \\
\pi (1)&=1.
\end{align*}
Thus, $\pi$ is a real *-algebra homomorphism. Moreover,
\begin{align*}
\pi_\cc((x+\iota y)+(x'+\iota y'))&=\pi(x+x')+\iota \pi(y+y')\\
&=(\pi(x)+\iota\pi(y))+(\pi(x')+\iota\pi(y'))\\
&=\pi_\cc(x+\iota y)+\pi_\cc(x'+\iota y')\,,\\
\pi_\cc((x+\iota y)(x'+\iota y'))&=\pi(xx'-yy')+\iota \pi(xy'+yx')\\
&=\pi(x)\pi(x')-\pi(y)\pi(y')+\iota \pi(x)\pi(y')+\iota\pi(y)\pi(x')\\
&=(\pi(x)+\iota\pi(y))(\pi(x')+\iota\pi(y'))\\
&=\pi_\cc(x+\iota y)\pi_\cc(x'+\iota y')\,,\\
\pi_\cc((x+\iota y)^c)&=\pi(x^c)+\iota\pi(y^c)\\
&=\pi(x)^c+\iota\pi(y)^c\\
&=\pi_\cc(x+\iota y)^c\,,\\
\pi_\cc(x+\iota y)&=x+\iota y\quad\mathrm{if\ }x,y\in\rr\,.
\end{align*}
Thus, $\pi_\cc$ is a surjective complex *-algebra homomorphism.
\end{proof}

We can easily study the effect of $\pi$ on the the map $\con$ of Definition~\ref{def:actiongen}.

\begin{remark}
Let $h,l \in \D\hh$. If $h$ is invertible, equality~\eqref{eq:action2} implies that
\[\pi(\con(h,l))=\con(h_1,l_1)=\con(\pi(h),\pi(l))\,.\]
If $h\in\epsilon\hh^*$, then
\[\pi(\con(h,l))=\pi(h_2^{-1}l h_2)= h_2^{-1}l_1 h_2=\con(h,\pi(l))\,.\]
If $h=0$ then both of the previous formulae are true.
\end{remark}

\begin{proposition}
For each stem function $F: D \longrightarrow \D\hh_\cc$, the function ${}^\pi \! F\vcentcolon=\pi_\cc\circ F$ is a stem function
\[{}^\pi \! F: D \longrightarrow \hh_\cc.\]
Moreover, if $F$ is holomorphic then ${}^\pi \! F$ is holomorphic too. Finally, the map $F\mapsto {}^\pi \! F$ is a *-algebra homomorphism.
\end{proposition}

\begin{proof}
By definition, for every $z \in D$
\[{}^\pi \! F(\overline{z})=\pi_\cc (F(\overline{z}))=\pi_\cc\big(\overline{F(z)}\big)=\pi(F_1(z))-\iota\pi(F_2(z))=\overline{\pi_\cc (F(z))}=\overline{{}^\pi \! F(z)}\,.\]
Thus, ${}^\pi \! F$ is a stem function. The fact that $F\mapsto {}^\pi \! F$ is a *-algebra homomorphism follows at once from the previous lemma.

By applying again the previous lemma, we observe that $\pi_\cc$ is a $\cc$-linear map, whence it coincides with its differential. Thus, if $F$ is holomorphic then
\[\frac{\partial {}^\pi \! F}{\partial \overline{z}}=\pi_\cc \circ \frac{\partial F}{\partial \overline{z}}\equiv0\,,\]
whence ${}^\pi \! F$ is a holomorphic function, too.
\end{proof}

The previous result allows us to give the next definition and to derive the subsequent corollary.

\begin{definition}
Let $\Omega=\Omega_D \subseteq \Q_{\D\hh}$ and let $f=\mathcal{I}(F)$ be a function in $\mathcal{S}(\Omega)$. We define the \emph{primal part} of $f$ as the quaternionic slice function ${}^\pi \! f \vcentcolon=\mathcal{I}({}^\pi \! F)\in\mathcal{S}(\Omega\cap\hh)$.
\end{definition}

\begin{corollary}\label{cor:piomo}
Let $\Omega=\Omega_D \subseteq \Q_{\D\hh}$. The map
\begin{align*}
\mathcal{S}(\Omega)&\to\mathcal{S}(\Omega\cap\hh)\\
f &\mapsto {}^\pi \! f
\end{align*}
is a *-algebra homomorphism. If $\Omega$ is open in $\Q_{\D\hh}$, the homomorphism maps slice regular functions into slice regular functions.
\end{corollary}

The entire construction is designed to satisfy the following property:

\begin{remark}\label{rmk:primal}
Let $\Omega=\Omega_D \subseteq \Q_{\D\hh}$ and consider a function $f \in \mathcal{S}(\Omega)$, induced by the stem function $F$. For all $x=\alpha+\beta J$ (whence $x_1=\alpha+\beta J_1$) and for $z=\alpha+\iota \beta$, it holds
\[\pi(f(x))=\pi (F_1(z))+ J_1 \pi (F_2(z))={}^\pi \! f(x_1).\]
As a consequence,
\begin{align*}
\pi(f^\circ_s(x))&=({}^{\pi}\!f)_s^\circ(x_1), \\
\pi(f'_s(x))&=({}^{\pi}\!f)_s'(x_1)\,.
\end{align*}
\end{remark}

The previous remark will be useful in Section~\ref{sec:zeros} to relate the zeros of a slice function over dual quaternions with the zeros of its primal part.

\begin{remark}
We can repeat the construction by using any *-algebra homomorphism $\psi$ instead of $\pi$, indeed
\[ \psi (f(\alpha + J \beta))=\psi(F_1(z)+JF_2(z)) = \psi (F_1(z))+\psi(J) \psi (F_2(z)),\]
where $\psi(J) \in \s_{\D\hh}$ because $n(\psi(J))=\psi(n(J))=\psi(1)=1$ and $t(\psi(J))=\psi(t(J))=\psi(0)=0$. In our case $\psi=\pi$, the map $J\mapsto\psi(J)$ for $J\in\s_{\D\hh}$ is the natural projection of $\s_{\D\hh}$ onto $\s_{\hh}$.
\end{remark}

Let us relate our construction with the concept of primal part of a polynomial considered in \cite{hegedus} (again, with a different convention about the side of the coefficients).

\begin{definition}
The map $\mathrm{primal}:\D\hh[t]\to\hh[t]$ is defined as
\[\mathrm{primal}\left(\sum_{n=0}^dt^na_n\right)\vcentcolon=\sum_{n=0}^dt^n \pi(a_n)\,.\]
\end{definition}

\begin{remark}\label{rmk:primalevaluation}
Let us consider $P(t)\in\D\hh[t]$ and evaluate it at $h=h_1+\epsilon h_2\in\D\hh$. Then
\[\pi(P(h))=\mathrm{primal}(P)(h_1)\,.\]
\end{remark}

\begin{proposition}
If $f\in\mathcal{S}(\Q_{\D\hh})$ is defined as $f:=P_{|_{\Q_{\D\hh}}}$ for some $P(t)\in\D\hh[t]$, then ${}^\pi \! f=\mathrm{primal}(P)$. As a consequence, $\mathrm{primal}:\D\hh[t]\to\hh[t]$ is a *-algebra homomorphism. 
\end{proposition}

\begin{proof}
If $P(t)=\sum_{n=0}^dt^na_n$ then $f(x)=\sum_{n=0}^dx^na_n$. By Lemma~\ref{lemma:homomorphism} and Remark~\ref{rmk:primal}, for all $x\in\D\hh$ it holds
\[{}^\pi \! f(\pi(x))=\pi(f(x))=\sum_{n=0}^d\pi(x)^n\pi(a_n)\,.\]
Since $\pi:\D\hh\to\hh$ is surjective, we conclude that
\[{}^\pi \! f(w)=\sum_{n=0}^dw^n\pi(a_n)=\mathrm{primal}(P)(w)\]
for all $w\in\hh$, which is our first statement.

The second statement now follows from Proposition~\ref{prop:algpol} and Corollary~\ref{cor:piomo}.
\end{proof}

\begin{remark}
\label{rkm:motionnorm}
If $P(t)\in\D\hh[t]$ is a motion polynomial of degree $d$, then $\mathrm{primal}(P)(t)\in\hh[t]$ has degree $d$. The norm of $\mathrm{primal}(P)(t)$ is a $2d$-degree real polynomial and it coincides with $N(P)(t)$.
\end{remark}


\section{Zeros of slice functions}\label{sec:zeros}

In this section, we describe some algebraic and geometric properties of the zero set
\[V(f)\vcentcolon=\{x\in \Omega\,:\,f(x)=0\}\]
of a slice function $f \in \mathcal{S}(\Omega)$ with $\Omega=\Omega_D$. Moreover, we study how this set is related to the zero sets of $f^c, {}^\pi \! f$ and $N(f)$.

Before proceeding towards the main results, let us establish two useful equalities.

\begin{lemma}\label{lemma:conjugate}
Let $f \in \mathcal{S}(\Omega)$. For all $x\in\Omega\cap\rr$, it holds $f^c(x)=f(x)^c$. For all $x \in \Omega\setminus\rr$ and all $y\in\s_x$, it holds $f^c(\con(f'_s(x),y^c))=f(y)^c$.
\end{lemma}
\begin{proof}
If $x \in \Omega \cap \rr$, then $f^c(x)=(f^c)^\circ_s(x)=f_s^\circ(x)^c=f(x)^c$.

Suppose, instead, $x \in \Omega\setminus\rr$ and $y\in\s_x$. We claim that $\con(h,l)h^c=h^cl$ for all $h\in\D\hh$ and we compute
\begin{align*}
f^c(\con(f'_s(x),y^c))&= f_s^\circ(x)^c+ \im(\con(f'_s(x),y^c))f'_s(x)^c= f_s^\circ(x)^c+\con(f'_s(x),\im(y^c))f'_s(x)^c \\
&=f_s^\circ(x)^c+f'_s(x)^c \im(y^c)=f_s^\circ(x)^c-f'_s(x)^c \im(y)=f(y)^c\,,
\end{align*}
which proves our thesis.

Our claim can be easily derived from the equality $h\con(h,l)=lh$ proven in Remark~\ref{rmk:conproprerty2}. If $h$ is invertible, it suffices to multiply each hand of the equality by $h^c$ both on the left and on the right and to divide it by the dual number $n(h)$. If, instead, $h=\epsilon h_2$, it suffices to multiply each hand of the equality by $h_2^c$ both on the left and on the right and to divide it by the real number $n(h_2)$.
\end{proof}

We are now ready to study the zero set $V(f)$ and its relation to $V(f^c)$. For each $x \in \Omega \setminus \rr$, let us set the notation
\[T_{x_1}\vcentcolon=\{x_1+\epsilon\gamma\,:\,\gamma \in \im(\hh), \ \gamma \perp \im(x_1)\}\]
for the tangent plane to the $2$-sphere $\s_x\cap\hh$ at $x_1$. This is a consistent extension of the notation $T_{J_1}$ we have already set for $J_1 \in \s_{\hh}$.

\begin{theorem}
\label{thm:zero}
Let $f \in \mathcal{S}(\Omega)$. If $x\in\Omega\cap\rr$, then $\s_x=\{x\}$ is included either in both $V(f)$ and in $V(f^c)$ or in none of the two. If instead $x \in \Omega\setminus\rr$, then one of the following properties holds:
\begin{enumerate}
\item $V(f)$ does not intersect $\s_x$;
\item $V(f) \cap \s_x=\{y\}$, $f'_s(x)$ is invertible and $y=\re(x)-f^\circ_s(x)f'_s(x)^{-1}$;
\item $V(f) \cap \s_x=T_{y_1}$ for some $y_1 \in \s_x \cap \hh$ and $f'_s(x), f^\circ_s(x) \in \epsilon \hh^*$;
\item $V(f)$ includes $\s_x$ and $f'_s(x)=f_s^\circ(x)=0$.
\end{enumerate}
In each of the aforementioned cases, respectively:
\begin{enumerate}
\item $V(f^c)$ does not intersect $\s_x$;
\item $V(f^c) \cap \s_x = \{\con(f'_s(x), y^c)\}$;
\item $V(f^c) \cap \s_x=T_{\con(f'_s(x),y_1^c)}$;
\item $V(f^c)$ includes $\s_x$.
\end{enumerate}
\end{theorem}
\begin{proof}
Our statement for $x\in\Omega\cap\rr$ follows from the equality $f^c(x)=f(x)^c$. Now let us suppose $x \in \Omega \setminus \rr$. For all $y \in \s_x$, the following decomposition holds:
\begin{equation}
\label{eq:rappr}
f(y)=f_s^\circ(x)+\im(y)f'_s(x)\,.
\end{equation}
If $V(f) \cap \s_x \neq \emptyset$, then consider $y \in V(f) \cap \s_x$.
\begin{itemize}
\item If $f'_s(x)$ is invertible then, by equality~\eqref{eq:rappr}, we have
\[\im(y)=-f_s^\circ(x)f'_s(x)^{-1}\,,\]
whence $y=\re(x)-f^\circ_s(x)f'_s(x)^{-1}$.
\item If instead $f'_s(x) \in \epsilon \hh^*$ then, starting again from equality~\eqref{eq:rappr}, we have
\[f_s^\circ(x)=-\im(y)f'_s(x)\,.\]
In this case, $f_s^\circ(x)\in\epsilon\hh^*$. Moreover, for every $z \in \s_x$, we decompose $f(z)$ as follows:
\[f(z)=f_s^\circ(x)+\im(z)f'_s(x)=(-\im(y)+\im(z))f'_s(x)=(z-y)f'_s(x)\,.\]
By Remark~\ref{rmk:zerodivisors}, we conclude that $f(z)=0$ if and only if, $z-y \in \epsilon \hh$. This is a same as $z_1=y_1$, i.e., $z \in T_{y_1}$.
\item Finally, if $f'_s(x)=0$, then $f \equiv f^\circ_s(x)$ in $\s_x$. Since $f(y)=0$, we immediately conclude that $f_s^\circ(x)=0$ and $V(f)\supseteq\s_x$.
\end{itemize}
If $x \in \Omega \cap \rr$ then $\s_x=\{x\}$. It is straightforward that either $V(f)\supseteq\s_x$ or $V(f) \cap \s_x=\emptyset$, depending on whether $f(x)$ vanishes or not.

We now prove the statement concerning $f^c$. If $x \in \Omega \cap \rr$, whence $\s_x=\{x\}$, then either $V(f^c)\supseteq\s_x$ or $V(f^c) \cap \s_x=\emptyset$, depending on whether $f^c(x)=f(x)^c$ vanishes or not. Now let $x \in \Omega \setminus \rr$. First suppose $V(f) \cap \s_x$ includes a point $y$. Then, by Lemma~\ref{lemma:conjugate}, $V(f^c) \cap \s_x$ includes the point $\con(f'_s(x),y^c)$. According to what we have proven so far, there are three possibilities.
\begin{itemize}
\item If $V(f) \cap \s_x=\{y\}$, then $(f^c)'_s(x)=f'_s(x)^c$ is invertible and $V(f^c) \cap \s_x = \{\con(f'_s(x), y^c)\}$. 
\item If $V(f) \cap \s_x=T_{y_1}$, then $(f^c)'_s(x)=f'_s(x)^c \in \epsilon \hh^*$ and $V(f^c) \cap \s_x=T_{\con(f'_s(x),y_1^c)}$.
\item If $V(f)\supseteq\s_x$ then $(f^c)'_s(x)=f'_s(x)^c=0$ and $V(f^c)\supseteq\s_x$.
\end{itemize}
The only remaining case is $V(f) \cap \s_x=\emptyset$. In this case, $f^c$ cannot have any zero $z\in\s_x$: otherwise, by Lemma~\ref{lemma:conjugate}, $f=(f^c)^c$ would vanish at $\con(f'_s(x)^c,z^c)\in\s_x$ and we would obtain a contradiction.
\end{proof}

The next result connects the zero set of a function $f \in \mathcal{S}(\Omega)$ with the zero set of its primal part.

\begin{proposition}\label{prop:zeroprimal}
Let $f \in \mathcal{S}(\Omega)$. If $x \in \Omega \cap \rr$, then $f(x) \in \epsilon \hh$, if and only if, ${}^\pi \! f(x)=0$. If $x \in \Omega \setminus \rr$, then there are three possibilities:
\begin{enumerate}
\item all values of $f$ in $\s_x$ are invertible and $V({}^\pi \! f)$ does not intersect $\s_x \cap \hh$;
\item $f$ maps exactly one tangent plane $T_{y_1}$ into $\epsilon \hh$ and $V({}^\pi \! f) \cap \s_x= \{y_1\}$;
\item $f$ maps $\s_x$ into $\epsilon \hh$ and $V({}^\pi \! f)$ includes $\s_x \cap \hh$.
\end{enumerate}
Moreover:
\begin{enumerate}
\item if $V(f) \cap \s_x=\{y\}$, then $V({}^\pi \! f) \cap \s_x= \{y_1\}$;
\item if $V(f) \cap \s_x=T_{y_1}$ or $V(f)\supseteq\s_x$, then $V({}^\pi \! f)$ includes $\s_x \cap \hh$.
\end{enumerate}
\end{proposition}
\begin{proof}
By Remark~\ref{rmk:primal}, for all $w\in\Omega$ it holds $f(w)\in\epsilon\hh$ if, and only if, ${}^\pi \! f(w_1)=0$.

If $w=x\in\Omega\cap\rr$, it holds $w_1=x$ and the first statement follows.
 
We can prove the second statement as follows. If $x\in\Omega\setminus\rr$, then ${}^\pi \! f$ may vanish:
\begin{enumerate}
\item at no point of $\s_x \cap \hh$;
\item at exactly one point of $\s_x \cap \hh$, say $y_1$;
\item at all points of $\s_x \cap \hh$.
\end{enumerate}
In case {1.} it holds $f(w)\not\in\epsilon\hh$ for all $w\in\s_x$. In case {2.} it holds $f(y_1+\epsilon\gamma)\in\epsilon\hh$ for all $\gamma \in \im(\hh), \gamma \perp \im(y_1)$, while $f(w)\not\in\epsilon\hh$ for all $w\in\s_x$ with $w_1\neq y_1$. In case {3.} it holds $f(w)\in\epsilon\hh$ for all $w\in\s_x$.

Let us now prove the third statement. If $V(f) \cap \s_x=\{y\}$ with $y \in \s_x$ then ${}^\pi \! f(y_1)=0$. Moreover, if ${}^\pi \! f$ had another zero $z_1 \in \s_x \cap \hh$  then $({}^{\pi}\!f)_s'(x_1)=\pi(f'_s(x))$ would vanish and $f_s'(x)$ would be a zero divisor, contradicting Theorem~\ref{thm:zero}. If, instead, $V(f) \cap \s_x=T_{y_1}$ or $V(f)\supseteq\s_x$, then by the same theorem, $f_s'(x)$ and $f^\circ_s(x)$ belong to $\epsilon\hh$. As a consequence, $f$ maps $\s_x$ into $\epsilon \hh$ and $V({}^\pi \! f) \supseteq \s_x \cap \hh$.
\end{proof}

We now study the zero set of the normal function $N(f)$, taking full advantage of the properties of $\D\hh$. We begin by establishing that it is circular.

\begin{proposition}\label{prop:zero-circular-nf}
Let $f \in \mathcal{S}(\Omega)$. Then $V(N(f))$ is circular, i.e., $V(N(f))\cap \s_x \neq \emptyset$ implies $V(N(f))\supseteq\s_x$.
\end{proposition}
\begin{proof}
Suppose $\s_x=\alpha+\beta \s_{\D\hh}$. For the sake of simplicity, we denote $f^\circ_s(x)$ and $\beta f'_s(x)$ by $a$ and $b$, respectively. For all $J\in\s_{\D\hh}$, by formula \eqref{eqn:normal},
\begin{align*}
N(f)(\alpha+\beta J)&=n(a)-n(b)+J t(ab^c)=n(a)-n(b)+J_1t(ab^c)+\epsilon J_2t(ab^c) \\
&=\underbrace{n(a)-n(b)}_{\in \rr+ \epsilon \rr}+\underbrace{J_1t(a_1b_1^c)}_{\in \im(\hh)}+\underbrace{\epsilon J_1 t(a_1b_2^c+a_2b_1^c)+ \epsilon J_2 t(a_1b_1^c)}_{\in \epsilon \im(\hh)}\,.
\end{align*}
If $N(f)(\alpha+\beta I)=0$ then $n(a)-n(b), I_1t(a_1b_1^c)$ and $\epsilon I_1 t(a_1b_2^c+a_2b_1^c)+ \epsilon I_2 t(a_1b_1^c)$ vanish, separately. This is, in turn, equivalent to $n(a)-n(b)=t(a_1b_1^c)=t(a_1b_2^c+a_2b_1^c)=0$. If this is the case, then $N(f)(\alpha+\beta J)=0$ independently of $J$ and $V(N(f))\supseteq\s_x$.
\end{proof}

Our next aim is studying the relation between $V(N(f))$ and $V(f)$. The next lemma will be useful to this end, because it connects the spherical derivative and the spherical value of $N(f)$ to the values of $f$.

\begin{lemma}\label{lemma:dernorm}
Let $f \in \mathcal{S}(\Omega)$ and let $x \in \Omega \setminus \rr$. For all $y \in \s_x$, the following equalities hold.
\begin{align*}
N(f)'_s(x)&=t(f(y)f'_s(x)^c)\,, \\
N(f)^\circ_s(x)&=f(y)f^\circ_s(x)^c-\im(y)f_s'(x)f(y)^c \,.
\end{align*}
\end{lemma}
\begin{proof}
For all $y= \alpha+\beta J \in \s_x$, we have 
\[f(y)=f_s^\circ(x)+\im(y)f'_s(x)=a+Jb\,,\]
where as usual we denote $f^\circ_s(x)$ and $\beta f'_s(x)$ by $a$ and $b$, respectively. By Equation \eqref{eqn:normal}, $N(f)'_s(x)=t(ab^c)$ and $N(f)^\circ_s(x)=n(a)-n(b)$. Recalling that $\D\rr$ is the center of $\D\hh$ and noticing that it is invariant under *-involution, the following equalities hold:
\begin{align*}
t(f(y)f'_s(x)^c)&=t((a+Jb)b^c)=t(ab^c+Jn(b)) \\
&=t(ab^c)+t(Jn(b))=t(ab^c)+Jn(b)-n(b)^cJ=t(ab^c)\,.
\end{align*}
Moreover, 
\begin{align*}
f(y)f^\circ_s(x)^c-\im(y)f_s'(x)f(y)^c&=(a+Jb)a^c-Jb(a+Jb)^c\\
&= n(a)+Jba^c-Jba^c+Jn(b)J=n(a)-n(b)\,.\qedhere
\end{align*}
\end{proof}

We are now able to state and prove the following theorem.

\begin{theorem}
\label{thm:zeronorm}
Let $f \in \mathcal{S}(\Omega)$ with $\Omega=\Omega_D$. Then
\begin{equation}\label{eq:zeronorm}
V(N(f))=\bigcup_{V(f) \cap \s_x\neq \emptyset} \s_x \ \ \cup  \bigcup_{\s_x \cap \hh \subseteq V({}^\pi \! f)} \s_x\,.
\end{equation} 
Moreover, for each $x\in\Omega\setminus\rr$ the normal function $N(f)$ vanishes (identically) in $\s_x$, if and only if, either $f$ has a unique zero in $\s_x$ or ${}^\pi \! f$ vanishes identically in $\s_x \cap \hh$.
\end{theorem}
\begin{proof}
As a first step, let us prove that $V(N(f))$ includes $\bigcup_{V(f) \cap \s_x\neq \emptyset} \s_x$. If $y \in V(f) \cap \s_x$ then, by Lemma~\ref{lemma:dernorm}, $N(f)'_s(x)=N(f)^\circ_s(x)=0$, whence $V(N(f))\supseteq\s_x$.

As a second step, we prove that $V(N(f))$ includes $\bigcup_{\s_x \cap \hh \subseteq V({}^\pi \! f)} \s_x$. If ${}^\pi \! f\equiv 0$ in $\s_x \cap \hh$ then $a=f_s^\circ(x)$ and $b=\beta f'_s(x)$ have $a_1=0=b_1$. In other words, $a$ and $b$ belong to $\epsilon \hh$, whence $n(a)=n(b)=ab^c=0$. As a consequence, for all $y=\alpha+\beta J \in \s_x$, the expression
\[N(f)(y)=n(a)-n(b)+Jt(ab^c)\]
vanishes.

As a third step, let us take any $\s_x$ contained in  
\[V(N(f)) \ \ \setminus \bigcup_{\s_x \cap \hh \subseteq V({}^\pi \! f)} \s_x \]
and prove that $\s_x$ includes exactly one zero of $f$. We observe that $N({}^\pi \! f)={}^\pi \! N(f)$ by Corollary~\ref{cor:piomo}, whence $V(N({}^\pi \! f))\supseteq\s_x\cap\hh$. By~\cite[Proposition 3.9]{librospringer}, the function ${}^\pi \! f$ has a zero $y_1=\alpha+\beta J_1$ in $\s_x \cap \hh$, a zero which is unique because we have assumed ${}^\pi \! f$ not to vanish identically in $\s_x\cap\hh$. We complete our proof by finding a unique zero of $f$ in $T_{y_1}$. If $a=f_s^\circ(x)$ and $b=\beta f'_s(x)$ (whence $\beta({}^\pi \! f)'_s(x)=b_1$ and $({}^\pi \! f)_s^\circ(x)=a_1=-J_1b_1$), we have to prove that there exists a unique $\gamma\in\im(\hh)$ with $\gamma\perp J_1$ such that
\begin{align*}
0=f(y_1+\epsilon\gamma)&=a+(J_1+\epsilon\gamma)b=a_1+J_1b_1+\epsilon(a_2+J_1b_2+\gamma b_1)=\epsilon(a_2+J_1b_2+\gamma b_1)\,.
\end{align*}
This happens if, and only if, $a_2b_1^c+J_1b_2b_1^c$ is an element of $\im(\hh)$, orthogonal to $J_1$, i.e.,
\begin{equation}\label{eq:normsystem}
\begin{cases}
t(a_2b_1^c+J_1b_2b_1^c)=0\\
t(J_1a_2b_1^c-b_2b_1^c)=0
\end{cases}
\end{equation}
After recalling that $t(pq)=t(qp)$ and $t(pq^c)=t(qp^c)$ for all $p,q\in\hh$, we can observe that
\[t(a_2b_1^c+J_1b_2b_1^c)=t(a_2b_1^c)+t(b_1^cJ_1b_2)=t(a_2b_1^c)+t(a_1^cb_2)=t(a_1b_2^c+a_2b_1^c)\]
is the dual part of $t(ab^c)$ and that
\[t(J_1a_2b_1^c-b_2b_1^c)=t(b_1^cJ_1a_2)-t(b_1^cb_2)=t(a_1^ca_2)-t(b_1^cb_2)\]
is the dual part of $n(a)-n(b)$. Since $N(f)$ vanishes identically in $\s_x$, we know that $n(a)-n(b)=0=t(ba^c)$. Thus, system~\eqref{eq:normsystem} is fulfilled and our proof is complete.
\end{proof}

We conclude the section with some examples that illustrate the previous results.

\begin{example}
For all $x\in\Q_{\D\hh}$, let
\[f(x)=x^2+1\,.\]
Then $f^c=f$ and $N(f)(x)=(x^2+1)^2$. Moreover, ${}^\pi \! f(x_1)=x_1^2+1$. Thus,
\[V(f)=V(f^c)=V(N(f))=\s_{\D\hh},\quad V({}^\pi \! f)=\s_\hh\,.\]
\end{example}

\begin{example}
For all $x\in\Q_{\D\hh}$, let
\[f(x)=x^2- x \epsilon +\epsilon i +1=(x-i)\cdot(x+i-\epsilon)\,.\]
It holds ${}^\pi \! f(x_1)=x_1^2+1$ and $N(f)=(x^2+1)(x^2+1-2x\epsilon)$, whence
\[V({}^\pi \! f)=\s_\hh,\quad V(N(f))=\s_{\D\hh}\,.\]
It is easy to observe that $i$ is a zero of $f$. Moreover, since $f(x)$ coincides with $- x \epsilon +\epsilon i$ for all $x\in\s_{\D\hh}$, we conclude that $f_s'(i)=-\epsilon$. Thus,
\[V(f)=T_{i},\quad V(f^c)=T_{-i}\,,\]
where we took into account the fact that $\con(-\epsilon , i^c)=-i$.
\end{example}

\begin{example}
For all $x\in\Q_{\D\hh}$, let
\[f(x)=x-1-i-\epsilon j\,.\]
By direct computation, $f^c(x)=x-1+i+\epsilon j$, $N(f)(x)=x^2-2x+2$ and ${}^\pi \! f(x_1)=x_1-1-i$. As a consequence,
\[V(f)=\{1+i+\epsilon j\},\quad V(f^c)=\{1-i-\epsilon j\},\quad V(N(f))=1+\s_{\D\hh},\quad V({}^\pi \! f)=\{1+i\}\,.\]
\end{example}

\begin{example}
On $\Q_{\D\hh}$, let
\[f \equiv \epsilon i\,.\]
By direct computation, $f^c\equiv-\epsilon i$, $N(f)\equiv 0$ and ${}^\pi \! f\equiv 0$. Thus,
\[V(f)=V(f^c)=\emptyset,\quad V(N(f))=\Q_{\D\hh},\quad V({}^\pi \! f)=\hh\,.\]
\end{example}

\section{Zeros of slice products}\label{sec:products}

This section describes in great detail the zero set of the slice product of two slice functions over $\D\hh$. We begin with a result that expresses the values of a slice product as products of values of its two factors.

\begin{theorem}
\label{thm:pointwiseprod}
Let $f,g \in \mathcal{S}(\Omega)$ and fix $x\in\Omega$. If $x\in \Omega \cap \rr$, then $(f\cdot g)(x)=f(x)g(x)$.

Suppose instead $x\in \Omega \setminus \rr$. If $y,z\in\s_x$ fulfill one of the following (mutually equivalent) conditions:
\begin{enumerate}
\item $yf(y)-f(y)z=0$;
\item $zf^c(z)-f^c(z)y=0$;
\end{enumerate}
then
\[(f\cdot g)(y)=f(y)g(z)\,.\]
Condition {\it 1.} is equivalent to $z=\con(f(y),y)$, when $f(y)$ is invertible; it is equivalent to $z\in T_{\con(f(y),y_1)}$, when $f(y)$ is a zero divisor; it is automatically fulfilled when $f(y)=0$. Similarly, condition {\it 2.} is equivalent to $y=\con(f^c(z),z)$ when $f^c(z)$ is invertible;  it is equivalent to $y\in T_{\con(f^c(z),z_1)}$ when $f^c(z)$ is a zero divisor;  it is automatically fulfilled when $f^c(z)=0$.
\end{theorem}

\begin{proof}
If $x\in \Omega \cap \rr$, then 
\[(f\cdot g)(x)=f^\circ_s(x)g^\circ_s(x)=f(x)g(x)=f(x)g(\con(f(x),x))=f(\con(f^c(x),x))\,g(x)\,.\]
Now suppose instead $x\in \Omega \setminus \rr$ and let $y,z\in\s_x$.
\begin{itemize}
\item From formulae~\eqref{eqn:repr} and~\eqref{eqn:prod}, it follows that
\[(f\cdot g)(y)=f(y)g^\circ_s(x)+\im(y)f(y)g'_s(x)\,.\]
This expression coincides with
\[f(y)g(z)=f(y)g^\circ_s(x)+f(y)\im(z)g'_s(x)\]
whenever $\im(y)f(y)=f(y)\im(z)$, which is equivalent to condition {\it 1.} 
\item We can prove the equivalence between conditions {\it 1.} and {\it 2.}, as follows. For the sake of simplicity, we denote $f^\circ_s(x)$ and $\beta f'_s(x)$ by $a$ and $b$, respectively. Supposing $y=\alpha+\beta J,z=\alpha+\beta K$ for some $J,K\in\s_{\D\hh}$, it holds $f(y)=a+Jb$ and, by formula \eqref{eqn:con}, $f^c(z)=a^c+Kb^c$. Condition {\it 1.} is equivalent to
\[0=J(a+Jb)-(a+Jb)K=-aK-b+J(a-bK)=(a-bK)(-K)+J(a-bK)\,.\]
The last equality is equivalent to
\[0=K(a^c+Kb^c)-(a^c+Kb^c)J\,,\]
which is, in turn, equivalent to condition {\it 2.}
\item The characterization of conditions {\it 1.} and {\it 2.} follows directly from Remark~\ref{rmk:conproprerty2}.
\end{itemize}
\end{proof}

\begin{corollary}\label{cor:pointwiseprod}
Let $f,g \in \mathcal{S}(\Omega)$. The formulae
\begin{align*}
(f\cdot g)(y)&=f(y)\,g(\con(f(y),y))\,,\\
(f\cdot g)(\con(f^c(z),z))&=f(\con(f^c(z),z))\,g(z)
\end{align*}
hold for all $y,z \in \Omega$. As a consequence, $V(f\cdot g)$ includes both the zero set $V(f)$ of $f$ and the set $\{\con(f^c(z),z)\,:\,z\in V(g)\}$.
\end{corollary}

To deepen the study of $V(f\cdot g)$, we recall Theorem~\ref{thm:zero}: the different types of zeros of $f\cdot g$ correspond to different properties of its spherical derivative $(f\cdot g)'_s$. Therefore, it is useful to establish the next result.

\begin{lemma}
\label{lemma:dersphprod}
Let $f,g \in \mathcal{S}(\Omega)$ and let $x \in \Omega \setminus \rr$.
\begin{enumerate}
\item If $y \in V(f) \cap \s_x$, then $(f\cdot g)'_s(x)=f'_s(x)g(\con(f'_s(x),y^c))$.
\item If $z \in V(g) \cap \s_x$, then $(f\cdot g)'_s(x)=(f^c(z))^c g'_s(x)$.
\item If $y \in V(f) \cap \s_x$ and $z \in V(g) \cap \s_x$, then 
$(f\cdot g)'_s(x)=f'_s(x)\big(\con(f'_s(x),y^c)-z\big)g'_s(x)$.
\end{enumerate}
\end{lemma}
\begin{proof}
According to formula~\eqref{eqn:prod}, it holds $(f\cdot g)'_s=f_s^\circ g'_s+f_s' g^\circ_s$.

If $f$ has a zero $y$ in $\s_x$, then $f_s^\circ(x)=-\im(y)f'_s(x)$ and
\begin{align*}
(f\cdot g)'_s(x)&=-\im(y)f'_s(x)g'_s(x)+f'_s(x)g^\circ_s(x) \\
&=f'_s(x)(-\con(f'_s(x),\im(y))g'_s(x)+g^\circ_s(x))\\
&=f'_s(x)g(\con(f'_s(x),y^c))\,,
\end{align*}
where we have taken into account Remark \ref{rmk:conproprerty2}.

If $g$ has a zero $z$ in $\s_x$, then $g_s^\circ(x)=-\im(z)g'_s(x)$ and
\begin{align*}
(f \cdot g)'_s(x)&= f_s^\circ(x) g'_s(x) + f'_s(x) (-\im(z) g'_s(x))\\
&=(f_s^\circ(x)- f'_s(x) \im(z))g'_s(x)\\
&=(f^c(z))^cg'_s(x)\,.
\end{align*}

Finally, if both $f$ and $g$ have zeros in $\s_x$, namely $y$ and $z$, then
\begin{align*}
(f\cdot g)'_s(x)&=-\im(y)f'_s(x) g'_s(x)-f'_s(x) \im(z) g'_s(x) \\
&=-\left(\im(y)f'_s(x)+f'_s(x)\im(z)\right)\,g'_s(x)\\
&=f'_s(x)\big(\con(f'_s(x),\im(y^c))-\im(z)\big)\,g'_s(x)\\
&=f'_s(x)\big(\con(f'_s(x),y^c)-z\big)g'_s(x)\,,
\end{align*}
where we have used again Remark \ref{rmk:conproprerty2}, along with the equalities $-\im(y)=\im(y^c)$ and $\re(y^c)=\re(y)=\re(z)$.
\end{proof}

\begin{theorem}\label{thm:prod}
Let $f,g \in \mathcal{S}(\Omega)$. If $x\in\Omega\cap\rr$ then $x\in V(f\cdot g)$ is equivalent to $x\in V(f)\cup V(g)\cup (V({}^\pi \! f)\cap V({}^\pi \! g))$. If, instead, $x \in \Omega \setminus \rr$, then the following statements hold.
\begin{enumerate}

\item If $V(f) \supseteq \s_x$ or $V(g) \supseteq \s_x$, then $V(f\cdot g)\supseteq\s_x$.

\item If $V(f) \cap \s_x = T_{y_1}$ and $V(g) \cap \s_x=T_{z_1}$, then $V(f\cdot g)\supseteq\s_x$.

\item Suppose $V(f) \cap \s_x = T_{y_1}$ and $V(g) \cap \s_x=\{z\}$.
\begin{itemize}
\item If $z_1=\con(f'_s(x),y_1^c)$, then $V(f\cdot g)\supseteq\s_x$.
\item Otherwise, $V(f\cdot g) \cap \s_x =T_{y_1}$.
\end{itemize}

\item Suppose $V(f) \cap \s_x = T_{y_1}$ and $V(g) \cap \s_x=\emptyset$.
\begin{itemize}
\item If $\con(f'_s(x), y_1^c)\in V({}^\pi \! g)$, then $V(f\cdot g)\supseteq\s_x$.
\item Otherwise, $V(f\cdot g) \cap \s_x =T_{y_1}$.
\end{itemize}

\item Suppose $V(f) \cap \s_x = \{y\}$ and $V(g) \cap \s_x=T_{z_1}$.
\begin{itemize}
\item If $z_1=\con({}^\pi \!f'_s(x_1),y_1^c)$, then $V(f\cdot g)\supseteq\s_x$.
\item Otherwise, $V(f\cdot g) \cap \s_x =T_{y_1}$.
\end{itemize}

\item Suppose $V(f) \cap \s_x = \{y\}$ and $V(g) \cap \s_x=\{z\}$.
\begin{itemize}
\item If $z=\con(f'_s(x),y^c)$, then $V(f\cdot g)\supseteq\s_x$.
\item If $z$ is a point of $T_{\con({}^\pi \! f'_s(x_1),y_1^c)}$ other than $\con(f'_s(x),y^c)$, then $V(f\cdot g) \cap \s_x =T_{y_1}$.
\item Otherwise, $V(f\cdot g) \cap \s_x =\{y\}$.
\end{itemize}

\item Suppose $V(f) \cap \s_x = \{y\}$ and $V(g) \cap \s_x=\emptyset$.
\begin{itemize}
\item If $\con({}^\pi \!f'_s(x_1), y_1^c)\in V({}^\pi \! g)$, then $V(f\cdot g) \cap \s_x =T_{y_1}$.
\item Otherwise, $V(f\cdot g) \cap \s_x =\{y\}$.
\end{itemize}

\item Suppose $V(f) \cap \s_x=\emptyset$ and $V(g) \cap \s_x=T_{z_1}$.
\begin{itemize}
\item If $z_1\in V({}^\pi \! f^c)$, then $V(f\cdot g)\supseteq\s_x$.
\item Otherwise, $V(f\cdot g) \cap \s_x =T_{\con({}^\pi \! f^c(z_1),z_1)}$.
\end{itemize}

\item Suppose $V(f) \cap \s_x=\emptyset$ and $V(g) \cap \s_x=\{z\}$.
\begin{itemize}
\item If $z_1\in V({}^\pi \! f^c)$, then $V(f\cdot g) \cap \s_x =T_{\con(f^c(z),z_1)}$.
\item Otherwise, $V(f\cdot g) \cap \s_x =\{\con((f^c(z),z)\}$.
\end{itemize}

\item If $V(f) \cap \s_x=\emptyset$ and $V(g) \cap \s_x=\emptyset$, then
\[V(f\cdot g) \cap \s_x =\{y\in\s_x\ :\ y_1\in V({}^\pi \! f),\ \con(f(y),y_1)\in V({}^\pi \! g)\}\,.\]
\end{enumerate}

\end{theorem}

\begin{proof}
 The first statement follows follows from the fact that
\[(f \cdot g)(x)=f(x)g(x)\]
for all $x\in\Omega\cap\rr$. To prove the second statement, we proceed step by step by step and we repeatedly apply Theorem~\ref{thm:zero}, Corollary~\ref{cor:pointwiseprod} and Lemma~\ref{lemma:dersphprod}.

\begin{enumerate}
\item If $V(f) \supseteq \s_x$ then $V(f\cdot g) \supseteq \s_x$. If $V(g) \supseteq \s_x$, take any $z\in\s_x$. Then $V(f\cdot g)\cap\s_x$ includes the point $\con(f^c(z),z)$ and $g'_s(x)=0$. Since
\[(f\cdot g)'_s(x)=(f^c(z))^c g'_s(x)=0\,,\]
it follows that $V(f\cdot g) \supseteq \s_x$.

\item If $V(f) \cap \s_x = T_{y_1}$ and $V(g) \cap \s_x=T_{z_1}$, then $V(f\cdot g) \supseteq T_{y_1}$ and $f_s'(x),g_s'(x)\in\epsilon\hh^*$. Since
\[(f\cdot g)'_s(x)=f'_s(x)\big(\con(f'_s(x),y_1^c)-z_1\big)g'_s(x)=0\,,\]
we immediately conclude that $V(f\cdot g)\supseteq\s_x$.

\item If $V(f) \cap \s_x = T_{y_1}$ and $V(g) \cap \s_x=\{z\}$, then $V(f\cdot g) \supseteq T_{y_1}$, $f_s'(x)\in\epsilon\hh^*$ and $g_s'(x)$ is invertible. Either $V(f\cdot g)\supseteq\s_x$ or $V(f\cdot g) \cap \s_x = T_{y_1}$, depending on whether $(f\cdot g)'_s(x)$ vanishes or not. But
\[(f\cdot g)'_s(x)=f'_s(x)\big(\con(f'_s(x),y_1^c)-z\big)g'_s(x)\,,\]
vanishes if, and only if, the second factor belongs to $\epsilon\hh$. This happens if, and only if, $z_1=\con(f'_s(x), y_1^c)$.

\item If $V(f) \cap \s_x = T_{y_1}$ and $V(g) \cap \s_x=\emptyset$, then $V(f\cdot g) \supseteq T_{y_1}$ and $f_s'(x)\in\epsilon\hh^*$. Either $V(f\cdot g)\supseteq\s_x$ or $V(f\cdot g) \cap \s_x = T_{y_1}$, depending on whether $(f\cdot g)'_s(x)$ vanishes or not. The expression
\[(f\cdot g)'_s(x)=f'_s(x)g(\con(f'_s(x),y_1^c))\,,\]
vanishes if, and only if, the second factor belongs to $\epsilon\hh$. This happens if, and only if, $\con(f'_s(x), y_1^c)\in V({}^\pi \! g)$.

\item If $V(f) \cap \s_x = \{y\}$ and $V(g) \cap \s_x=T_{z_1}$, then $V(f\cdot g) \supseteq \{y\}$, $f_s'(x)$ is invertible and $g_s'(x)\in\epsilon\hh^*$. The expression
\[(f\cdot g)'_s(x)=f'_s(x)\big(\con(f'_s(x),y^c)-z_1\big)g'_s(x)\,.\]
vanishes if, and only if, the second factor belongs to $\epsilon\hh$; this is, in turn, equivalent to $z_1=\con({}^\pi \!f'_s(x_1),y_1^c)$. If this is the case, then $V(f\cdot g)\supseteq\s_x$. Otherwise, $(f\cdot g)'_s(x)\in\epsilon\hh^*$ and $V(f\cdot g) \cap \s_x = T_{y_1}$.

\item If $V(f) \cap \s_x = \{y\}$ and $V(g) \cap \s_x = \{z\}$, then $V(f\cdot g) \supseteq \{y\}$ and $f_s'(x),g_s'(x)$ are invertible. By the equality
\[(f\cdot g)'_s(x)=f'_s(x)\big(\con(f'_s(x),y^c)-z\big)g'_s(x)\,,\]
there are three possibilities. If $z=\con(f'_s(x),y^c)$ then $(f\cdot g)'_s(x)=0$ and $V(f\cdot g)\supseteq\s_x$. If not, but if we still have $z_1=\con({}^\pi \!f'_s(x_1),y_1^c)$, then $(f\cdot g)'_s(x)\in \epsilon \hh^*$ and $V(f\cdot g) \cap \s_x =T_{y_1}$. Otherwise, $(f\cdot g)'_s(x)$ is invertible and $V(f\cdot g) \cap \s_x = \{y\}$.

\item If $V(f) \cap \s_x = \{y\}$ and $V(g) \cap \s_x=\emptyset$, then $V(f\cdot g) \supseteq \{y\}$ and $f_s'(x)$ is invertible. Either $V(f\cdot g) \cap \s_x = T_{y_1}$ or $V(f\cdot g) \cap \s_x = \{y\}$, depending on whether or not 
\[(f\cdot g)'_s(x)=f'_s(x)g(\con(f'_s(x),y^c))\]
belongs to $\epsilon\hh^*$. This happens if, and only if, $g(\con(f'_s(x),y^c))\in\epsilon\hh^*$, which is, in turn, equivalent to $\con({}^\pi \!f'_s(x_1), y_1^c)\in V({}^\pi \! g)$.

\item If $V(f) \cap \s_x=\emptyset$ and $V(g) \cap \s_x=T_{z_1}$, then $V(f\cdot g) \supseteq \{\con(f^c(z_1),z_1)\}$ and $g_s'(x)\in\epsilon\hh^*$. The expression
\[(f\cdot g)'_s(x)=(f^c(z_1))^c g'_s(x)\]
vanishes if, and only if, $f^c(z_1)\in\epsilon\hh$; this is, in turn, equivalent to $z_1\in V({}^\pi \! f^c)$. If this is the case, then $V(f\cdot g)\supseteq\s_x$. Otherwise, $(f\cdot g)'_s(x)\in\epsilon\hh^*$ and $V(f\cdot g) \cap \s_x = T_{\con({}^\pi \! f^c(z_1),z_1)}$.

\item If $V(f) \cap \s_x=\emptyset$ and $V(g) \cap \s_x=\{z\}$, then $V(f\cdot g) \cap \s_x \supseteq \{\con(f^c(z),z)\}$ and $g_s'(x)$ is invertible. The expression
\[(f\cdot g)'_s(x)=(f^c(z))^c g'_s(x)\]
belongs to $\epsilon\hh^*$ if, and only if, the first factor does. This is, in turn, equivalent to $z_1\in V({}^\pi \! f^c)$. If this is the case, then $V(f\cdot g) \cap \s_x =T_{\con(f^c(z),z_1)}$. Otherwise, $V(f\cdot g) \cap \s_x =\{\con(f^c(z),z)\}$.

\item If $V(f) \cap \s_x=\emptyset$ and $V(g) \cap \s_x=\emptyset$, then for all $y\in\s_x$ it holds
\[(f\cdot g)(y)=f(y)\,g(\con(f(y),y))\,.\]
This product vanishes if, and only if, both factors belong to $\epsilon\hh^*$. This happens if, and only if, $y_1\in V({}^\pi \! f)$ and $\con(f(y),y_1)\in V({}^\pi \! g)$.
\end{enumerate}
\end{proof}

\begin{table}
\begin{center}
\begin{tabular}{|c||c|c|c|c|}
\hline
\boldmath$f \cdot g$ & \boldmath$\emptyset$ & \boldmath$z$ & \boldmath$T_{z_1}$ & \boldmath$\mathbb{S}_x$\\
\hline
\hline
\rule[-4mm]{0mm}{1cm}
\boldmath$\emptyset$ & $\emptyset \vee \{\star\} \vee T_{\star} \vee \mathbb{S}_{x}$ & $\{\con(f^c(z),z)\} \vee T_{\con(f^c(z),z_1)}$ & $T_{\con({}^\pi \! f^c(z_1),z_1)} \vee \mathbb{S}_x$ & $\mathbb{S}_x$ \\ 
\hline
\rule[-4mm]{0mm}{1cm}
\boldmath$y$ &  $\{y\} \vee T_{y_1}$ & $\{y\} \vee T_{y_1} \vee \mathbb{S}_x$ & $T_{y_1} \vee \mathbb{S}_x$ & $\mathbb{S}_x$\\
\hline
\rule[-4mm]{0mm}{1cm}
\boldmath$T_{y_1}$ & $T_{y_1} \vee \mathbb{S}_x $ & $T_{y_1} \vee\mathbb{S}_x$ & $\mathbb{S}_x$ & $\mathbb{S}_x$\\
\hline
\rule[-4mm]{0mm}{1cm}
\boldmath$\mathbb{S}_x$ & $\mathbb{S}_x$ & $\mathbb{S}_x$ & $\mathbb{S}_x$ & $\mathbb{S}_x$\\
\hline
\end{tabular}
\end{center}
\caption{Scheme from Theorem~\ref{thm:prod}. The first column indicates the nature of the zeros of $f$ in $\s_x$, the first row the nature of the zeros of $g$ in $\s_x$. Each crossing lists the possible intersections between the zero set of $f\cdot g$ and $\s_x$, without mentioning the conditions that distinguish the various possibilities.}
\label{tab:zeriprod}
\end{table}

\begin{examples}[Case {\it 6.}]
Fix $J \in \s_{\D\hh}$. For all $x\in \Q_{\D\hh}$, let
\[f(x)=x-i,\quad g(x)=x-J\,.\]
Then $V(f)=\{i\}$, $V(g)=\{J\}$ and $f'_s(i)=1={}^{\pi}\!f'_s(i)$. Let us determine the zeros of the product
\[(f\cdot g)(x)=x^2-x(i+J)+iJ\,.\]
Since $V({}^{\pi}\!f)=\{i\}$ and $V({}^{\pi}\!g)=\{J_1\}$ are both included in $\s_{\D\hh}$, the zero set $V(f\cdot g)$ must also be included in $\s_{\D\hh}$. There are three possibilities.
\begin{itemize}
\item If $J=\con(1, -i)=-i$ then $V(f\cdot g)=\s_{\D\hh}$.
\item If $J$ is a point of $T_{-i}$ other than $-i$, then $V(f\cdot g)=T_i$.
\item Otherwise, $V(f\cdot g)=\{i\}$.
\end{itemize}
\end{examples} 

\begin{example}[Case {\it 9.}]
For all $x\in \Q_{\D\hh}$, let
\[f(x)=2x-2i+\epsilon i,\quad g(x)=x+i+\epsilon j\,.\]
The function $f$ has no zeros in $\Q_{\D\hh}$, while $V(g)=\{-i-\epsilon j \}\subset\s_{\D\hh}$. Let us determine the zeros of the product
\[(f \cdot g)(x)=2x^2+x(\epsilon i+2\epsilon j)+2-2\epsilon k-\epsilon\,.\]
For $\s_{\D\hh}$, we compute $f^c(x)=2x+2i-\epsilon i$, whence $V({}^{\pi}\!f^c)=\{-i\}$, and
\[\con(f^c(-i-\epsilon j),-i)=\con(-\epsilon(i+2j),-i)=(-i-2j)^{-1}(-i)(-i-2j)=\frac{3}{5}i-\frac{4}{5}j\,.\]
Thus, $V(f\cdot g)\cap\s_{\D\hh}=T_{\frac{3}{5}i-\frac{4}{5}j}$. Since $V({}^{\pi}\!f)=\{i\}$ and $V({}^{\pi}\!g)=\{-i\}$ are both included in $\s_{\D\hh}$, there are no other possible zeros of $f\cdot g$ and
\[V(f\cdot g)=T_{\frac{3}{5}i-\frac{4}{5}j}\,.\]
\end{example}

\begin{example}[Cases {\it 4.} and {\it 10.}]
For all $x\in \Q_{\D\hh}$, let
\[f(x)=x\epsilon-2\epsilon k,\quad g(x)=x-i+\epsilon i\,.\]
Then $V(f)=T_{2k}$ and $g$ never vanishes in $\Q_{\D\hh}$. Let us determine the zero set of the product
\[(f \cdot g)(x)=x^2\epsilon-x\epsilon (i+2k)+2\epsilon j\,.\]
We first focus on $2\s_{\D\hh}\ni2k$: we compute $f'_s(2k)=\epsilon$ and $\con(\epsilon,2k)=2k$, which does not belong to $V({}^{\pi}\!g)=\{i\}$; we conclude that
\[V(f\cdot g)\cap2\s_{\D\hh}=T_{2k}\,.\]
Since $V({}^{\pi}\!f)=\hh$ and $V({}^{\pi}\!g)=\{i\}$, the other possible zeros $y\in\Q_{\D\hh}$ of $f\cdot g$ are determined by the equation $\con(f(y),y_1)=i$, which is equivalent to 
\[y_1=\con(f^c(i),i)=\con(\epsilon(i+2k),i)=(i+2k)^{-1}i(i+2k)=-\frac35 i+\frac45 k\,.\]
We conclude that
\[V(f\cdot g)=T_{2k}\cup T_{-\frac35 i+\frac45 k}\,.\]
\end{example}


\section{Zeros of slice regular functions}\label{sec:factorization}

In this section, we study the discreteness of the zeros of slice regular functions. We also show that they can be factored out, a fact which will be particularly useful for subsequent applications to motion polynomials.

Our discreteness results require two preliminary definitions.

\begin{definition}
If $\Omega=\Omega_D$, where $D$ is an open connected subset of $\cc$ that intersects the real line $\rr$ and is preserved by complex conjugation, then $\Omega$ is called a \emph{slice domain}. If $\Omega=\Omega_D$, where $D$ is an open subset of $\cc$ that does not intersect $\rr$ and has two connected components swapped by complex conjugation, then $\Omega$ is called a \emph{product domain}.
\end{definition}

It is not restrictive to study slice regular functions  on either slice domains or product domains, because all circular open subsets of $\D\hh$ are unions of slice domains and product domains. We now come to the announced discreteness results, where $\cc_J^+:=\{\alpha+J\beta:\alpha,\beta\in\rr,\beta>0\}$ and $\cc_J^-:=\{\alpha+J\beta:\alpha,\beta\in\rr,\beta<0\}$.

\begin{theorem}
Let $\Omega=\Omega_D$, where $D$ is an open subset of $\cc$, and let $f\in\mathcal{SR}(\Omega)$. If $\Omega$ is a slice domain then either
\begin{itemize}
\item[(1)] for each $J\in\s_{\D\hh}$, the intersection $V(f)\cap\cc_J$ is closed and discrete in $\Omega_J$; or
\item[(2)] $f$ vanishes identically.
\end{itemize}
If $\Omega$ is a product domain then, in addition to cases (1) and (2), there is one further possibility:
\begin{itemize}
\item[(3)] There exists $J\in\s_{\D\hh}$ such that
\begin{equation}\label{eq:fin}
f_{|_{\Omega\cap\cc_J^+}}\equiv0\mathrm{\ and\ }V(f)\cap\cc_J^-\mathrm{\ is\ closed\ and\ discrete\ in\ }\Omega_J.
\end{equation}
If ${}^{\pi}\!f'_s\not \equiv0$ then $J$ is unique; if instead ${}^{\pi}\!f'_s\equiv0$, then the imaginary units having the same property are exactly the elements of $T_{J_1}$. For all other $K\in\s_{\D\hh}$, the intersection $V(f)\cap\cc_K$ is closed and discrete in $\Omega_K$.
\end{itemize}
\end{theorem}

\begin{proof}
After applying~\cite[Theorem 4.11]{gpsalgebra}, it only remains to study the uniqueness of $J$ in case {\it (3)} by finding the solutions $J'$ of the equation $(J'-J)f'_s\equiv0$. If at least one value of $f'_s$ is invertible (that is, ${}^{\pi}\!f'_s\not \equiv0$), then the only solution is $J'=J$. Otherwise, the image of $f'_s$ is included in the set $\epsilon\hh$ and ${}^{\pi}\!f'_s\equiv0$. In this situation, the solutions are the $J'\in\s_{\D\hh}$ such that $J'-J\in\epsilon\hh$ or, equivalently, the elements of $T_{J_1}$.
\end{proof}

\begin{theorem}
Let $f\in\mathcal{SR}(\Omega)$, where either
\begin{itemize}
\item $\Omega\subseteq Q_{\D\hh}$ is a slice domain and ${}^{\pi}\!f\not\equiv0$; or
\item $\Omega\subseteq Q_{\D\hh}$ is a product domain and $N({}^{\pi}\!f)\not\equiv0$
\end{itemize}
Then the zero set $V(f)$ is a union of singletons $\{y\}$, tangent planes $T_{y_1}$ or ``spheres'' $\s_y$, each isolated from the rest of $V(f)$. If, moreover, $f=P_{|_{\Q_{\D\hh}}}$ for some polynomial $P(t)\in\D\hh[t]$, then the union is finite.
\end{theorem}

\begin{proof}
By Theorem~\ref{thm:zero}, the zero set $V(f)$ is a union of singletons $\{y\}$, tangent planes $T_{y_1}$ or ``spheres'' $\s_y$. Moreover, each singleton $\{y\}$ included in $V(f)$ corresponds to a singleton $\{y_1\}$ included in $V({}^{\pi}\!f)$; each tangent plane $T_{y_1}$ or ``sphere'' $\s_y$ corresponds to a $2$-sphere $\s_{y_1}$ included in $V({}^{\pi}\!f)$.

We claim that, under either of our two hypotheses, $V({}^{\pi}\!f)$ consists of isolated points or isolated $2$-spheres of the form $\s_{y_1}$, whence the first statement immediately follows.

As for the second statement, the equality $f=P_{|_{\Q_{\D\hh}}}$ implies that ${}^{\pi}\!f$ coincides with the quaternionic polynomial $\mathrm{primal}(P)(t)$, which has finitely many isolated zeros or $2$-spheres of zeros (see, e.g., \cite[Proposition 3.30]{librospringer}).

We prove our previous claim as follows.
\begin{itemize}
\item If $\Omega$ is a slice domain in $Q_{\D\hh}$ then $\Omega\cap\hh$ is a slice domain in $\hh$. If ${}^{\pi}\!f\not\equiv0$, then $V({}^{\pi}\!f)$ consists of isolated points $y_1$ or isolated $2$-spheres of the form $\s_{y_1}$ by~\cite[Theorem 3.12]{librospringer}.
\item If $\Omega$ is a product domain in $Q_{\D\hh}$ then $\Omega\cap\hh$ is a product domain in $\hh$. If $N({}^{\pi}\!f)\not\equiv0$, then, according to~\cite[page 1681]{perotti}, $V({}^{\pi}\!f)$ consists of isolated points $y_1$ or isolated $2$-spheres of the form $\s_{y_1}$.
\end{itemize}
This completes the proof.
\end{proof}

All examples provided in Sections~\ref{sec:zeros} and~\ref{sec:products} are slice regular functions on slice domains and they have zero sets of the types described in the last two theorems. We now provide a few pathological examples that do not fulfill the hypotheses of these theorems.

\begin{example}
Consider the slice regular functions
\[f(x)=1+\frac{\im(x)}{|\im(x)|}i,\quad g(x)=f(x)\cdot\epsilon=\epsilon+\frac{\im(x)}{|\im(x)|}\epsilon i\]
on the product domain $Q_{\D\hh}\setminus\rr=\Omega_{\cc^+}$: for each $J\in\s_{\D\hh}$, it holds $f_{|_{\cc_J^+}}\equiv1+Ji$ and $g_{|_{\cc_J^+}}\equiv(1+Ji)\epsilon$. As a consequence, 
\[V(f)=\cc_i^+,\quad V(g)=\bigcup_{J\in T_i}\cc_J^+\,.\]
We point out that ${}^{\pi}\!f'_s\not\equiv0\not\equiv{}^{\pi}\!f$ and $N({}^{\pi}\!f)\equiv0$, while ${}^{\pi}\!g'_s,{}^{\pi}\!g,N({}^{\pi}\!g)$ all vanish identically.
\end{example}

We now proceed forward, aiming at factoring out the zeros of slice regular functions.

\begin{definition}
Let $\Omega=\Omega_D\subseteq Q_{\D\hh}$. For $f,h\in\mathcal{S}(\Omega)$, we say that $h$ is a \emph{left factor} of $f$ if there exists $g\in\mathcal{S}(\Omega)$ such that
\[f(x)=(h\cdot g)(x)\]
in $\Omega$. If this is the case, we also say that $h(x)$ \emph{divides} $f(x)$ \emph{on the left} and write $h(x)\,|\,f(x)$. If $h(x)$ does not divide $f(x)$ on the left, we write $h(x){\not|}\ f(x)$.
\end{definition}

We adopt analogous terminologies and notations for other algebras, such as the algebra of quaternionic slice functions $\mathcal{S}(\Omega\cap\hh)$ and the algebra of polynomials over dual quaternions  $\D\hh[t]$.

\begin{theorem}\label{thm:factor}
Let $f\in\mathcal{SR}(\Omega)$ with $\Omega=\Omega_D$ and let $y\in\Omega$.
\begin{itemize}
\item The zero set $V(f)$ includes $y$ if, and only if, $x-y\,|\,f(x)$. If this is the case, then we have $x^2-xt(y)+n(y)\,|\,N(f)(x)$.
\end{itemize}
If $y\in\Omega\setminus\rr$ then 
\begin{itemize}
\item $V(f)\supseteq T_{y_1}$ if, and only if, $x-y\,|\,f(x)$ and $x_1^2-x_1t(y)+n(y)\,|\,{}^{\pi}\!f(x_1)$\,;
\item $V(f)\supseteq\s_y$ if, and only if, $x^2-xt(y)+n(y)\,|\,f(x)$.
\end{itemize}
\end{theorem}

\begin{proof}
We begin with the first statement.
\begin{itemize}
\item If $x-y\,|\,f(x)$ then $f(y)=0$ by Corollary~\ref{cor:pointwiseprod}. Now let us prove the converse implication. Let $J\in\s_{\D\hh}$ be an imaginary unit such that $y\in\Omega_J$ and let $\{1,J,e^J_1, Je^J_1,e^J_{2},Je^J_{2},e^J_{3},Je^J_{3}\}$ be an associated splitting basis of $\D\hh$.
By Lemma~\ref{lemma:splitting}, there exist holomorphic functions $f_0,f_1,f_2,f_3:\Omega_J \longrightarrow \cc_J$ such that 
\[f_{|\Omega_J}=\sum_{l=0}^{4} f_le^J_l\,,\]
where $e^J_0\vcentcolon =1$. If $f(y)=0$ then, for each $l\in\{0,1,2,3\}$, it holds $f_l(y)=0$ and there exists a holomorphic function $g_l:\Omega_J \longrightarrow \cc_J$ such that
\[f_l(z)=(z-y)g_l(y)\]
for all $z\in\Omega_J$. 
\\{\bf Claim 1.} {\it There exists $g\in\mathcal{SR}(\Omega)$ with $g_{|\Omega_J}=\sum_{l=0}^{4} g_le^J_l$.}
\\{\bf Claim 2.} {\it $f(x)=(x-y)\cdot g(x)$ in $\Omega$.}\\
We postpone the proofs of our claims until later and we complete the proof of the first statement by computing
\[N(f)(x)=(x-y)\cdot N(g)(x)\cdot(x-y^c)=(x-y)\cdot(x-y^c)\cdot N(g)(x)=(x^2-xt(y)+n(y))\cdot N(g)(x)\,.\]
\end{itemize}
We now suppose $y\in\Omega\setminus\rr$ and prove the third and second statement, reversing the order of presentation for the sake of clarity.
\begin{itemize}
\item $V(f)\supseteq \s_y$ is equivalent to $f(y)=0=f'_s(y)$ by Theorem~\ref{thm:zero}. By what we have already proven, $f(y)=0$ is equivalent to the existence of $g\in\mathcal{SR}(\Omega)$ with $f(x)=(x-y)\cdot g(x)$. By formula~\eqref{eqn:prod}, the last equality implies
\[f'_s(y)=1\,g^\circ_s(y)+\im(y^c)\,g'_s(y)=g(y^c)\,.\]
Thus, $f(y)=0=f'_s(y)$ is equivalent to the existence of $g,h\in\mathcal{SR}(\Omega)$ with $f(x)=(x-y)\cdot g(x)$ and with $g(x)=(x-y^c)\cdot h(x)$. This is, in turn, equivalent to asking for
\[(x-y)\cdot(x-y^c)=x^2-xt(y)+n(y)\]
to divide $f(x)$.
\item $V(f)\supseteq T_{y_1}$ is equivalent to $f(y)=0,f'_s(y)\in\epsilon\hh^*$ by Theorem~\ref{thm:zero}. This is equivalent to the existence of $g\in\mathcal{SR}(\Omega)$ with $f(x)=(x-y)\cdot g(x)$ and with ${}^{\pi}\!g(y_1^c)=0$ (by Remark~\ref{rmk:primal}). This is, in turn, equivalent to asking for $x-y$ to divide $f(x)$ and for
\[(x_1-y_1)\cdot(x_1-y_1^c)=x_1^2-x_1t(y_1)+n(y_1)=x_1^2-x_1t(y)+n(y)\]
to divide ${}^{\pi}\!f(x_1)$.
\end{itemize}
{\bf Proof of claim 1.} Setting 
\begin{align*}
G_1(\alpha+i\beta)&:=\frac12\sum_{l=0}^{4} (g_l(\alpha+J\beta)+g_l(\alpha-J\beta))e^J_l\\
G_2(\alpha+i\beta)&:=\frac{J}2\sum_{l=0}^{4} (g_l(\alpha-J\beta)-g_l(\alpha+J\beta))e^J_l
\end{align*}
for all $\alpha+i\beta\in D$ defines a stem function $G=G_1+\iota G_2:D\to\D\hh$. $G$ is holomorphic by direct inspection, whence it induces a slice regular function $g=\I(G)\in\mathcal{SR}(\Omega)$. Now, for all $z=\alpha+J\beta\in\Omega_J$,
\[g(z)=G_1(\alpha+i\beta)+JG_2(\alpha+i\beta)=\sum_{l=0}^{4} g_l(z)e^J_l\,,\]
as desired.
\\{\bf Proof of claim 2.} We wish to prove that $h(x):=(x-y)\cdot g(x)$ coincides with $f(x)$ in $\Omega$. As explained in Section~\ref{sec:functions}, it suffices to prove that they coincide in $\Omega_J$. Now, for all $z\in\Omega_J$, Theorem~\ref{thm:pointwiseprod} guarantees that $h(z)=(z-y)g(z)$. Thus,
\[h(z)=(z-y)g(z)=(z-y)\sum_{l=0}^{4} g_l(z)e^J_l=\sum_{l=0}^{4} f_l(z)e^J_l=f(z)\,,\]
as desired.
\end{proof}

\begin{corollary}\label{cor:uniquefactor}
Let $f\in\mathcal{SR}(\Omega)$ with $\Omega=\Omega_D$ and let $y\in\Omega\setminus\rr$. If
\[x^2-xt(y)+n(y)\,|\,N(f)(x),\mathnormal{\ but\ } x_1^2-x_1t(y)+n(y){\not|}\ {}^{\pi}\!f(x_1)\,,\]
then $f$ has a unique zero $w\in\s_y$ and $x-w\,|\,f(x)$. If
\[x_1^2-x_1t(y)+n(y)\,|\,{}^{\pi}\!f(x_1),\mathnormal{\ but\ } x^2-xt(y)+n(y){\not|}\ f(x)\,,\]
then either $x-w{\not|}\ f(x)$ for all $w\in\s_y$ or the subset of those $w\in\s_y$ such that $x-w\,|\,f(x)$ is a tangent plane $T_{h_1}\subseteq\s_y$. Finally, if 
\[x^2-xt(y)+n(y)\,|\,f(x)\,,\]
then $x-w\,|\,f(x)$ for all $w\in\s_y$.
\end{corollary}

\begin{proof}
Let us prove our first statement. According to Theorem~\ref{thm:factor}, $V(N(f))$ includes $\s_y$ but $V({}^{\pi}\!f)$ does not include $\s_y\cap\hh$. By Theorem~\ref{thm:zeronorm}, $f$ has a unique zero $w\in\s_y$. We easily conclude that $x-w\,|\,f(x)$ by a further application of Theorem~\ref{thm:factor}.

The second statement follows directly from Theorem~\ref{thm:factor}.

As for the third statement, it follows from the fact that $x-w\,|\,x^2-xt(y)+n(y)$ for all $w\in\s_y$.
\end{proof}

The previous corollary is false when $y\in\Omega\cap\rr$, as proven by the next examples.

\begin{example}
Let us fix $y=y_1\in\rr$ and $v\in\im(\hh)$. Define $f\in\mathcal{SR}(\Q_{\D\hh})$ by the formula
\[f(x)=x-y_1-\epsilon v\,.\]
Then ${}^{\pi}\!f(x_1)=x_1-y_1$ and $N(f)(x)=x^2-2xy_1+y_1^2=x^2-xt(y)+n(y)$. Nevertheless, in $\s_y=\{y\}$ it holds $f(y)\neq0$ and $x-y{\not|}\ f(x)$.
\end{example}

\begin{example}
Let us fix $y=y_1\in\rr$ and define $f\in\mathcal{SR}(\Q_{\D\hh})$ by the formula
\[f(x)=x^2+x(\epsilon-t(y))+n(y)-y\epsilon\,.\]
Then ${}^{\pi}\!f(x_1)=x_1^2-x_1t(y)+n(y)$ and $x^2-xt(y)+n(y){\not|}\ f(x)$. Nevertheless, in $\s_y=\{y\}$ it holds $f(y)=0$ and $x-y\,|\,f(x)$.
\end{example}

The difference between the non real case and the real case is explained by the next remark, which involves the concept \emph{characteristic polynomial} of a dual quaternion $h\in\D\hh$: a polynomial having minimal degree among all monic real polynomial vanishing at $h$.

\begin{remark}\label{rkm:delta}
Let us fix $h\in\D\hh$. The norm of $t-h$, namely
\[\Delta_h(t):=(t-h)\cdot(t-h^c)=t^2-t\,t(h)+n(h)\]
is a real polynomial if, and only if, $h\in(\rr+\epsilon\im(\hh))\cup\Q_{\D\hh}$. Moreover:
\begin{itemize}
\item If $h\in\Q_{\D\hh}\setminus\rr$, then $\Delta_h(t)$ is the characteristic polynomial of each element of its zero set $\s_h$.
\item If $h\in\rr+\epsilon\im(\hh)$, then $\Delta_h(t)=(t-h_1)^2$ is the characteristic polynomial of each element of its zero set $h_1+\epsilon\hh$ except $h_1$, whose characteristic polynomial is $t-h_1$.
\end{itemize}
Finally, for each monic quadratic real polynomial $M(t)$: if $M(t)$ is irreducible in $\rr[t]$, then it equals $\Delta_h(t)$ for some $h\in\Q_{\D\hh}\setminus\rr$; if $M(t)=(t-h_1)^2$ with $h_1\in\rr$, then $M(t)=\Delta_{h_1}(t)$; if $M(t)$ has two real roots, then it has no other roots in $\D\hh$ and it does not coincide with $\Delta_h(t)$ for any $h\in\D\hh$.
\end{remark}


\section{Applications to the study of motion polynomials}\label{sec:applications}

This section explains the meaning of possible factorizations of the motion polynomial $P(t)$ in formula~\eqref{eq:trajectory} and it studies existence and uniqueness of such factorizations. This is done combining material from~\cite{hegedus} with applications of our new results about the zeros of slice functions over dual quaternions.

Throughout this section, we only consider monic polynomials. There is no loss of generality in doing so, because multiplying the motion polynomial $P(t)$ in formula~\eqref{eq:trajectory} by the inverse of its leading coefficient will only result in a change in the coordinate frame.

\subsection{Meaning of factorization of motion polynomials}

Let us consider a monic linear polynomial
\[P(t)=t-h\]
with $h\in\D\hh$. $P(t)$ is a motion polynomial if, and only if, $N(P)(t)=\Delta_h(t)$ belongs to $\rr[t]$. This is, in turn, equivalent to $h\in(\rr+\epsilon\im(\hh))\cup\Q_{\D\hh}$. We distinguish two cases for the transformations
\[1+\epsilon x \longmapsto \frac{(t-h)\,(1+\epsilon x)\,(t-\tilde h)}{\Delta_h(t)}=1+\epsilon\left((t-h_1)x(t-h_1)^{-1}+2h_2(t-h_1)^{-1}\right)\,.\]
\begin{description}
\item[Translations.] If $h\in \rr+\epsilon\im(\hh)$, then the previous transformation is a translation with translation vector $2h_2(t-h_1)^{-1}$. The direction of the vector does not depend on $t$, although its length does: the trajectories are straight lines parallel to $h_2$.
\item[Rotations.] If $h\in \Q_{\D\hh}\setminus\rr$, then the previous transformation is a rotation. The rotation axis, which has Pl\"ucker coordinates $\left(\frac{\im(h_1)}{|\im(h_1)|},\frac{h_2}{|\im(h_1)|}\right)$, does not depend on $t$. The rotation angle $\theta$ is determined by the equality $\cos(\frac\theta2)=\frac{t-\re(h_1)}{|t-h_1|}$. Thus, the trajectories are circles whose centers lie on the fixed axis.
\end{description}

Motions of former type have been identified in~\cite{hegedus} as those of linkages consisting of a single prismatic joint. Similarly, motions of the latter type are associated to linkages consisting of a single revolute joint. Let us give an explicit example of the latter type.

\begin{example}
For all $t \in \rr,$ consider the monic linear polynomial
\[P(t)=t-i+\epsilon j\,.\]
It is a motion polynomial because $N(P)(t)=t^2+1\in \rr[t]$. For each $t_0 \in \rr$, we consider the proper rigid transformation 
\[1+\epsilon x \longmapsto \dfrac{(t_0-i+\epsilon j)(1+\epsilon x)(t_0+i+\epsilon j)}{t_0^2+1}\,,\]
which is a rotation around the axis with Pl\"ucker coordinates $(i,-j)$ (i.e., the axis with direction $i$ through the point $-i\wedge j=-k$). For instance, the trajectory of $1+\epsilon k$, corresponding to the point $(0,0,1)$ in $\rr^3$, is the rational curve
\[t\longmapsto 1+\epsilon\dfrac{4tj+(t^2-3)k}{t^2+1}\,.\]
It is a parametrization of the circle of radius $2$ centered at $(0,0,-1)$ in the plane of points $(x_1,x_2,x_3)$ with $x_1=0$, except for the point $(0,0,1)$, which is the limit of the curve as $t\to\pm\infty$.
\end{example}

Now let us consider a motion polynomial $P(t) \in \D\hh[t]$ of degree $n>0$. Suppose $P(t)$ admits a factorization
\begin{equation}\label{eq:fact}
P(t)=(t-h^{(1)})\cdot(t-h^{(2)})\cdot \ \dots \ \cdot (t-h^{(n)})\quad \mathrm{with\ }h^{(1)},\ldots,h^{(n)}\in(\rr+\epsilon\im(\hh))\cup\Q_{\D\hh}\,.
\end{equation} 
Then the transformation in formula~\eqref{eq:actionpoly} is the composition of a number $s\in\{1,\ldots,n\}$ of pure rotations and of $n-s$ pure translations, completely determined by the dual quaternions $h^{(1)},\dots,h^{(n)}$ and by the parameter $t_0$. The resulting trajectories have been identified in~\cite{hegedus} as the motions of linkages consisting of $s$ revolute joints and $n-s$ prismatic joints. This motivated the search for sufficient conditions on a motion polynomial $P(t) \in \D\hh[t]$ that guarantee the existence of a factorization of the form~\eqref{eq:fact}. 

\subsection{Sufficient conditions for the existence of factorizations}

The work~\cite{hegedus} proved that the following property is a sufficient condition for the existence of a factorization of the form~\eqref{eq:fact}, with $h^{(1)},\ldots,h^{(n)}\in\Q_{\D\hh}\setminus\rr$ (corresponding to a linkage with $n$ revolute joints). 

\begin{definition}
A motion polynomial $P(t)\in \D\hh[t]$ is called \emph{generic} if every real polynomial that divides $\mathrm{primal}(P)(t)$ is constant.
\end{definition}

The key ingredient in the proof of this sufficient condition was~\cite[Lemma 3]{hegedus}, which we can restate and prove as follows (again, with a different convention about the side of the coefficients of polynomials in $\D\hh[t]$).

\begin{lemma}\label{lemma:commonsol}
Let $P(t)\in\D\hh[t]$ and let $y\in(\rr+\epsilon\im(\hh))\cup\Q_{\D\hh}$. If
\[\Delta_y(t)\,|\,N(P)(t),\quad \Delta_y(t){\not|}\ \mathrm{primal}(P)(t)\]
then $P(t)$ and $\Delta_y(t)$ have a unique common root $h$ and $t-h\,|\,P(t)$. Moreover, $h\in(\rr+\epsilon\im(\hh))\cup\Q_{\D\hh}$.
\end{lemma}

\begin{proof}
If $y\in\Q_{\D\hh}\setminus\rr$, then the thesis immediately follows from Corollary~\ref{cor:uniquefactor}. If, instead, $y\in\rr+\epsilon\im(\hh)$, then the results of Section~\ref{sec:factorization} do not generally apply. Therefore, we provide a direct proof of our thesis.

Our first step is proving that $P(t)$ and $\Delta_y(t)$ have a unique common zero $h$. Since $\Delta_y(t)=\Delta_{y_1}(t)=(t-y_1)^2$ divides $N(P)(t)$, Corollary~\ref{cor:pointwiseprod} tells us that $N(P)(t)$ vanishes at the real point $y_1$. By Theorem~\ref{thm:zeronorm}, so does $\mathrm{primal}(P)(t)$. Thus,
\begin{equation}\label{eq:reallinearfactor}
t-y_1\,|\,\mathrm{primal}(P)(t),\quad (t-y_1)^2{\not|}\ \mathrm{primal}(P)(t)\,.
\end{equation}
We have $P(t)=(t-y_1)\cdot Q(t)+r$ for some other polynomial $Q(t)\in\D\hh[t]$ and some $r\in\D\hh$. In this situation, $\mathrm{primal}(P)(t)=(t-y_1)\cdot \mathrm{primal}(Q)(t)+r_1$ and conditions~\eqref{eq:reallinearfactor} imply that $r_1=0$, while $\mathrm{primal}(Q)(y_1)=q\in\hh^*$. As a consequence, $P(t)=(t-y_1)\cdot Q(t)+\epsilon r_2$ and, for all $v\in\hh$,
\begin{align*}
P(y_1+\epsilon v)&=(t-y_1)\cdot Q(t)_{|t=y_1+\epsilon v}+\epsilon r_2\\
&=\epsilon v\,Q(y_1+\epsilon v)+\epsilon r_2\\
&=\epsilon v\,\mathrm{primal}(Q)(y_1)+\epsilon r_2\\
&=\epsilon(vq+r_2)\,,
\end{align*}
where the second equality follows by applying Remark~\ref{rmk:evaluation} to $t-y_1\in\rr[t]\subset\D\rr[t]$ and the third equality follows from Remark~\ref{rmk:primalevaluation}. As a consequence, the unique zero of $P(t)$ in the zero set $y_1+\epsilon\hh$ of $(t-y_1)^2=\Delta_y(t)$ is the point
\[h:=y_1-\epsilon r_2q^{-1}\,.\]

As a second step, we prove that $h$ belongs to $y_1+\epsilon \im(\hh)$, i.e., that $t(r_2q^{-1})=0$. By construction,
\begin{align*}
&N(P)(t)=(t-y_1)^2\cdot N(Q)(t)+\epsilon\cdot(t-y_1)\cdot R(t),\\
&R(t):=r_2\cdot \mathrm{primal}(Q)^c(t)+\mathrm{primal}(Q)(t)\cdot r_2^c\,.
\end{align*}
The hypothesis $(t-y_1)^2\,|\,N(P)(t)$ implies that $t-y_1\,|\,R(t)$. We conclude that
\[0=R(y_1)=r_2\,\mathrm{primal}(Q)^c(y_1)+\mathrm{primal}(Q)(y_1)\,r_2^c=r_2q^c+qr_2=t(r_2q^c)\,,\]
where the second and third equalities follow by applying Remark~\ref{rmk:evaluation} at $y_1\in\rr\subset\D\rr$. As a consequence, $t(r_2q^{-1})=\frac{t(r_2q^c)}{n(q)}=0$, as desired.

Our third step is proving that $t-h\,|\,P(t)$. To do so, let us start again from the equality $P(t)=(t-y_1)\cdot Q(t)+\epsilon r_2$. By Remark~\ref{rmk:evaluation},
\[Q(y_1)=\mathrm{primal}(Q)(y_1)=q\,,\]
whence there exists $\widetilde Q(t)\in\D\hh[t]$ such that $Q(t)=(t-y_1)\cdot \widetilde Q(t)+q$. Taking into account that $(t-y_1)^2=\Delta_h(t)$, we conclude that
\[P(t)=\Delta_{h}(t)\cdot \widetilde Q(t)+(t-y_1)\cdot q+\epsilon r_2=\Delta_{h}(t)\cdot \widetilde Q(t)+(t-y_1+\epsilon r_2q^{-1})\cdot q=\Delta_{h}(t)\cdot \widetilde Q(t)+(t-h)\cdot q\,.\]
Since $t-h\,|\,\Delta_h(t)$ by definition, our thesis $t-h\,|\,P(t)$ immediately follows.
\end{proof}

The same work~\cite{hegedus} hinted it was possible to generalize the sufficient condition for the existence of a factorization of the form~\eqref{eq:fact} to allow both revolute and prismatic joints. This is done in the next definition and in the subsequent result, which uses both Lemma~\ref{lemma:commonsol} and the Quaternionic Fundamental Theorem of Algebra (see, e.g.,~\cite[Theorem 3.18]{librospringer}).

\begin{definition}
A motion polynomial $P(t)\in \D\hh[t]$ is called \textit{$\Delta$-free} if 
\[\Delta_{y}(t){\not|}\ \mathrm{primal}(P)(t)\]
for all $y\in (\rr+\epsilon\im(\hh))\cup\Q_{\D\hh}$; or, equivalently, for all $y=y_1 \in \hh$.
\end{definition}

\begin{theorem}\label{thm:factgen}
Let $P(t)\in \D\hh[t]$ be a $\Delta$-free monic motion polynomial of degree $n>0$. Then $P(t)$ admits a factorization $P(t)=(t-h^{(1)})\cdot \ \dots \ \cdot (t-h^{(n)})$ with $h^{(1)},\ldots,h^{(n)}\in (\rr+\epsilon\im(\hh))\cup\Q_{\D\hh}$.
\end{theorem}

\begin{proof}
We proceed by induction on the degree $n$. For the case $n=1$, we already observed by direct computation that $P(t)=t-h^{(1)}$ with $h^{(1)}\in (\rr+\epsilon\im(\hh))\cup\Q_{\D\hh}$. Now let us suppose our thesis proven for all degrees $1,\ldots,n-1$ and let us establish it for degree $n$.

According to the Quaternionic Fundamental Theorem of Algebra, the quaternionic polynomial $\mathrm{primal}(P)$ admits a zero $y_1\in\hh$. Thus,
\[t-y_1\,|\,\mathrm{primal}(P)(t)\,.\]
As a consequence, $\Delta_{y_1}(t)$ divides the norm of $\mathrm{primal}(P)(t)$, which coincides with $N(P)(t)$ according to Remark~\ref{rkm:motionnorm}. By the definition of $\Delta$-free motion polynomial, $\Delta_{y_1}(t){\not|}\ \mathrm{primal}(P)(t)$. According to Lemma~\ref{lemma:commonsol}, 
\[P(t)=(t-h^{(1)})\cdot Q(t)\]
for some polynomial $Q(t)\in\D\hh[t]$ of degree $n-1$ and some point $h^{(1)} \in (\rr+\epsilon\im(\hh))\cup\Q_{\D\hh}$ that is a zero of $\Delta_{y_1}(t)$. By direct computation,
\begin{align*}
&N(P)(t)=\Delta_{h^{(1)}}(t)\cdot N(Q)(t)\,,\\
&\mathrm{primal}(P)(t)=(t-h^{(1)}_1)\cdot\mathrm{primal}(Q)(t)\,,
\end{align*}
whence $Q(t)$ is still a $\Delta$-free motion polynomial. Our inductive hypothesis guarantees the existence of a factorization $Q(t)=(t-h^{(2)})\cdot \ \dots \ \cdot (t-h^{(n)})$ with $h^{(2)},\ldots,h^{(n)}\in (\rr+\epsilon\im(\hh))\cup\Q_{\D\hh}$. It immediately follows that $P(t)$ admits a factorization of the form~\eqref{eq:fact}, as desired.
\end{proof}

\begin{remark}
A $\Delta$-free motion polynomial $P(t)$ of degree $n$ has as many different factorizations as the quaternionic polynomial $\mathrm{primal}(P)(t)$ does. This follows by direct inspection in the previous proof. In particular, $P(t)$ has at most $n!$ distinct factorizations.
\end{remark}

Techniques to classify all possible factorizations of the quaternionic polynomial $\mathrm{primal}(P)(t)$ are described in~\cite[\S 3.5]{librospringer}. We present here two significant examples.

\begin{example}
Let $P(t)$ be the following $\Delta$-free motion polynomial:
\[P(t)=(t+2j-\epsilon k)\cdot (t+i+\epsilon k)\,.\]
Its primal part 
\[\mathrm{primal}(P)(t)=(t+2j)\cdot(t+i)\]
has exactly two roots, namely $-2j$ and $\frac{3}{5}i-\frac{4}{5}j$. As a consequence of~\cite[Theorem 3.24]{librospringer}, $\mathrm{primal}(P)$ admits exactly two factorizations, the second one being
\[\mathrm{primal}(P)(t)=\bigg(t-\frac{3}{5}i+\frac{4}{5}j\bigg)\cdot\bigg(t+\frac{8}{5}i+\frac{6}{5}j\bigg)\]
If we repeat the proof of Theorem~\ref{thm:factgen} starting with this second factorization of $\mathrm{primal}(P)(t)$, we find a second factorization of $P(t)$, namely
\[P(t)=\bigg(t-\frac{3}{5}i+\frac{4}{5}j-\epsilon k\bigg)\cdot\bigg(t+\frac{8}{5}i+\frac{6}{5}j+\epsilon k\bigg)\,.\]
\end{example}

\begin{example}
The $\Delta$-free motion polynomial 
\[P(t)=(t-i+\epsilon j) \cdot (t-j+\epsilon k).\]
admits a unique factorization. Indeed, its primal part 
\[\mathrm{primal}(P)(t)=(t-i)\cdot (t-j)\]
only vanishes at $\{i\}$ and has a unique factorization by~\cite[Proposition 3.23]{librospringer}.
\end{example}

\subsection{Factorization in general}

In this final subsection, we treat the problem of factorization of motion polynomials in general. In addition to the results of the previous subsection, we will make use of the next lemma.

\begin{lemma}\label{lemma:commonsol2}
Let $P(t)\in\D\hh[t]$. Suppose $y\in(\rr+\epsilon\im(\hh))\cup\Q_{\D\hh}$ is such that $\Delta_y(t)\,|\,\mathrm{primal}(P)(t)$. The set of common zeros of $\Delta_y(t)$ and $P(t)$, which is the set of zeros $h$ of $\Delta_y(t)$ such that $t-h\,|\,P(t)$, may be:
\begin{enumerate}
\item the whole zero set of $\Delta_y(t)$;
\item a tangent plane $T_{w_1}\subset\s_y$, provided $y\in\Q_{\D\hh}\setminus\rr$;
\item the empty set.
\end{enumerate}
\end{lemma}

\begin{proof}
If $y\in\Q_{\D\hh}\setminus\rr$, then the thesis immediately follows from Corollary~\ref{cor:uniquefactor}. Let us therefore suppose that $y\in\rr+\epsilon\im(\hh)$, whence the zero set of $\Delta_y(t)=(t-y_1)^2$ is $y_1+\epsilon\hh$.

Let us divide $P(t)$ by $\Delta_y(t)$: $P(t)=\Delta_y(t)Q(t)+R(t)$ for some $Q(t),R(t)\in\D\hh[t]$ with $\deg R(t)\leq1$. If $R(t)$ is constant then we are either in case {\it 1.} or in case {\it 3.} Suppose, instead, $R(t)=ta+b$ with $a\neq0$.
The hypothesis $\Delta_y(t)\,|\,\mathrm{primal}(P)(t)$ implies that $\mathrm{primal}(R)(t)=0$. Thus, $R(t)=t\epsilon a_2+\epsilon b_2$ with $a_2\in\hh^*$ and
\[R(t)=(t+b_2a_2^{-1})\cdot\epsilon a_2=(t+b_2a_2^{-1})\epsilon a_2\]
has a unique quaternionic zero, namely $-b_2a_2^{-1}$. If this zero coincides with $y_1$, then the zero set of $R(t)=(t-y_1)\epsilon a_2$ is $y_1+\epsilon\hh$. Moreover, the equality $R(t)=(t-h)\epsilon a_2$ holds for all $h\in y_1+\epsilon\hh$, whence $t-h\,|\,P(t)$, and we are in case {\it 1.} If, instead, $-b_2a_2^{-1}$ does not coincide with $y_1$ then, for every $h\in y_1+\epsilon\hh$, the polynomial $R(t)$ does not vanish at $h$ and $t-h{\not|}\ R(t)$. Thus, we are in case {\it 3.}
\end{proof}

\begin{example}
The motion polynomials
\[P_1(t)=t^2+1,\quad P_2(t)=t^2+1+(t-i)\epsilon j,\quad P_3(t)=t^2+1+\epsilon j\]
have
\[\mathrm{primal}(P_\ell)(t)=t^2+1=\Delta_i(t),\quad N(P_\ell)(t)=(t^2+1)^2\]
for all $\ell\in\{1,2,3\}$. The zero sets of $P_1(t),P_2(t),P_3(t)$ are, respectively, $\s,T_i,\emptyset$.
\end{example}

We are now ready for a general description of motion polynomials. To this end, the next definition will be useful

\begin{definition}
A motion polynomial $P(t)\in\D\hh[t]$ is \textit{$\Delta$-covered} if, for each zero $y_1\in\hh$ of $\mathrm{primal}(P)(t)$, it holds
\[\Delta_{y_1}(t)\,|\,\mathrm{primal}(P)(t)\,.\]
\end{definition}

\begin{proposition}
Let $P(t)\in\D\hh[t]$ be a monic motion polynomial of degree $n>0$. Then 
\[P(t)=Q(t)\cdot R(t)\,,\]
where $Q(t)$ is a $\Delta$-free motion polynomial of degree $q\leq n$ and $R(t)$ is a $\Delta$-covered motion polynomial. The polynomial $Q(t)$ admits at least $1$ and at most $q!$ factorizations of the form
\[Q(t)=\left(t-h^{(1)}\right)\cdot\left(t-h^{(2)}\right)\cdot \ \dots \ \cdot \left(t-h^{(q)}\right)\quad \mathrm{with\ }h^{(1)},\ldots,h^{(q)}\in(\rr+\epsilon\im(\hh))\cup\Q_{\D\hh}\,,\]
each corresponding to exactly one factorization $\mathrm{primal}(Q)(t)=\left(t-h_1^{(1)}\right)\cdot\left(t-h_1^{(2)}\right)\cdot \ \dots \ \cdot \left(t-h_1^{(q)}\right)$ of its primal part.
The polynomial $R(t)$ admits either infinitely many or no factorization of the form
\[R(t)=\left(t-h^{(q+1)}\right)\cdot\left(t-h^{(q+2)}\right)\cdot \ \dots \ \cdot \left(t-h^{(n)}\right)\quad \mathrm{with\ }h^{(q+1)},\ldots,h^{(n)}\in(\rr+\epsilon\im(\hh))\cup\Q_{\D\hh}\,.\]
If it does, then $\Delta_{h_1^{(\ell)}}\,|\,\mathrm{primal}(R)(t)$ for all $\ell\in\{q+1,q+2,\ldots,n\}$.
\end{proposition}

\begin{proof}
Consider the quaternionic polynomial $\mathrm{primal}(P)(t)$: thanks to the techniques described in~\cite[\S 3.5]{librospringer}, it is possible to find a factorization
\[\mathrm{primal}(P)(t)=\left(t-p_1^{(1)}\right)\cdot\left(t-p_1^{(2)}\right)\cdot \ \dots \ \cdot \left(t-p_1^{(n)}\right)\]
with $p_1^{(1)},\ldots,p_1^{(n)}\in\hh$ such that, for some $q\leq n$,
\begin{align*}
&\Delta_{p_1^{(\ell)}}{\not|}\ \mathrm{primal}(P)(t)\mathrm{\ for\ }\ell\in\{1,2,\ldots,q\}\\
&\Delta_{p_1^{(\ell)}}\,|\,\mathrm{primal}(P)(t)\mathrm{\ for\ }\ell\in\{q+1,q+2,\ldots,n\}\,.
\end{align*}
Lemma~\ref{lemma:commonsol} applies to $P(t)$ at $p_1^{(1)}$, so that 
\[P(t)=(t-h^{(1)})\cdot P^{(1)}(t)\]
for some $h^{(1)}\in(\rr+\epsilon\im(\hh))\cup\Q_{\D\hh}$ with $h_1^{(1)}=p_1^{(1)}$ and for some $P^{(1)}(t)\in\D\hh[t]$ with
\[\mathrm{primal}(P^{(1)})(t)=\left(t-p_1^{(2)}\right)\cdot \ \dots \ \cdot \left(t-p_1^{(n)}\right)\,.\]
After $q-1$ further applications of Lemma~\ref{lemma:commonsol}, we find that
\[P(t)=\left(t-h^{(1)}\right)\cdot\left(t-h^{(2)}\right)\cdot \ \dots \ \cdot \left(t-h^{(q)}\right)\cdot P^{(q)}(t)\]
with $h^{(1)},\ldots,h^{(q)}\in(\rr+\epsilon\im(\hh))\cup\Q_{\D\hh}$ and with
\[\mathrm{primal}(P^{(q)})(t)=\left(t-p_1^{(q+1)}\right)\cdot \ \dots \ \cdot \left(t-p_1^{(n)}\right)\,.\]
If we set $R(t):=P^{(q)}(t)$, then the first and second statements are proven. 
To prove the third statement, we observe that if $R(t)$ admits a factorization of the form
\[R(t)=\left(t-h^{(q+1)}\right)\cdot\left(t-h^{(q+2)}\right)\cdot \ \dots \ \cdot \left(t-h^{(n)}\right)\quad \mathrm{with\ }h^{(q+1)},\ldots,h^{(n)}\in(\rr+\epsilon\im(\hh))\cup\Q_{\D\hh}\,,\]
then it admits infinitely many because of Lemma~\ref{lemma:commonsol2}. Furthermore, the previous equality implies that
\[\mathrm{primal}(R)(t)=\left(t-h_1^{(q+1)}\right)\cdot\left(t-h_1^{(q+2)}\right)\cdot \ \dots \ \cdot \left(t-h_1^{(n)}\right)\,.\]
As a consequence, the zero set of $\mathrm{primal}(R)(t)$ intersects $\s_{h_1^{(\ell)}}\cap\hh$ for all $\ell\in\{q+1,q+2,\ldots,n\}$. Since $R(t)$ is $\Delta$-covered by construction, it follows that $\Delta_{h_1^{(\ell)}}\,|\,\mathrm{primal}(R)(t)$ for all $\ell\in\{q+1,q+2,\ldots,n\}$.
\end{proof}

\begin{example}
The motion polynomial
\[P(t)=t^3+t^2(i+j+\epsilon i)-t\epsilon(1+k)\]
has $\mathrm{primal}(P)(t)=t^3+t^2(i+j)=(t+i+j)\cdot t^2$ and $N(P)(t)=(t^2+2) \cdot t^4$. It factorizes as
\[P(t)=(t+i+j)\cdot R(t),\quad R(t)=t^2+t\epsilon i=(t+\epsilon a)\cdot(t+\epsilon (i-a)) \mathit{\ for\ all\ } a\in\im(\hh)\,.\]
\end{example}

\begin{example}
The motion polynomial
\[P(t)=(t^2+1)^2+(t-i)\epsilon j\]
has $\mathrm{primal}(P)(t)=(t^2+1)^2$ and $N(P)(t)=(t^2+1)^4$. It holds
\[P(t)=(t-i)\cdot R(t),\quad R(t)=(t^2+1)(t+i)+\epsilon j\]
and $R(t)$ never vanishes in $(\rr+\epsilon\im(\hh))\cup\Q_{\D\hh}$.
\end{example}




\end{document}